\documentclass{amsart}
\usepackage{amsmath}
\usepackage{amsfonts}

\newtheorem{theorem}{Theorem}
\theoremstyle{plain}

\newtheorem{corollary}{Corollary}

\newtheorem{definition}{Definition}
\newtheorem{example}{Example}

\newtheorem{lemma}{Lemma}

\newtheorem{proposition}{Proposition}

\numberwithin{equation}{section}
\input{tcilatex}

\begin{document}
\title[Bilinear Multipliers of Small Lebesgue spaces]{Bilinear Multipliers
of Small Lebesgue spaces}
\author{\"{O}znur Kulak}
\address{Amasya University, Faculty of Arts and Sciences, Department of
Mathematics, Amasya, Turkey}
\email{oznur.kulak@amasya.edu.tr}
\urladdr{}
\thanks{.}
\author{A.Turan G\"{u}rkanl\i }
\address{Istanbul Arel University, Faculty of Science and Letters,
Department of Mathematics and Computer Sciences, \.{I}stanbul, Turkey}
\curraddr{\"{y} }
\email{turangurkanli@arel.edu.tr}
\urladdr{}
\thanks{}
\dedicatory{}
\thanks{}

\begin{abstract}
Let $G$ be a locally compact abelian metric group with Haar measure $\lambda 
$ and $\hat{G}$ its dual with Haar measure $\mu ,$ and $\lambda \left(
G\right) $ is finite. Assume that$~1<p_{i}<\infty $, $p_{i}^{\prime }=\frac{%
p_{i}}{p_{i}-1}$, $\left( i=1,2,3\right) $ and $\theta \geq 0$. Let $%
L^{(p_{i}^{\prime },\theta }\left( G\right) ,$ $\left( i=1,2,3\right) $ be
small Lebesgue spaces. A bounded measurable function $m\left( \xi ,\eta
\right) $ defined on $\hat{G}\times \hat{G}$ is said to be a bilinear
multiplier on $G$ of type $\left[ (p_{1}^{\prime };(p_{2}^{\prime
};(p_{3}^{\prime }\right] _{\theta }$ if the bilinear operator $B_{m}$
associated with the symbol $m$ 
\begin{equation*}
B_{m}\left( f,g\right) \left( x\right) =\dsum\limits_{s\in \hat{G}%
}\dsum\limits_{t\in \hat{G}}\hat{f}\left( s\right) \hat{g}\left( t\right)
m\left( s,t\right) \left\langle s+t,x\right\rangle
\end{equation*}%
defines a bounded bilinear operator from $L^{(p_{1}^{\prime },\theta }\left(
G\right) \times L^{(p_{2}^{\prime },\theta }\left( G\right) $ into $%
L^{(p_{3}^{\prime },\theta }\left( G\right) $. We denote by \ $BM_{\theta }%
\left[ (p_{1}^{\prime };(p_{2}^{\prime };(p_{3}^{\prime }\right] $ the space
of all bilinear multipliers of type $\left[ (p_{1}^{\prime };(p_{2}^{\prime
};(p_{3}^{\prime }\right] _{\theta }$. In this paper, we discuss some basic
properties of the space $BM_{\theta }\left[ (p_{1}^{\prime };(p_{2}^{\prime
};(p_{3}^{\prime }\right] $ and give examples of bilinear multipliers.
\end{abstract}

\maketitle

\section{\protect\bigskip Introduction}

Let $\Omega $ be locally compact Hausdorff space and let $\left( \Omega ,%
\mathit{B},\mu \right) $ be finite Borel measure space. The grand Lebesgue
space $L^{p)}\left( \Omega \right) $ is defined by the norm%
\begin{equation*}
\left\Vert f\right\Vert _{p)}=\underset{0<\varepsilon \leq p-1}{\sup }\left(
\varepsilon \underset{\Omega }{\dint }\left\vert f\right\vert
^{p-\varepsilon }d\mu \right) ^{\frac{1}{p-\varepsilon }}
\end{equation*}%
where $1<p<\infty $, $\left[ 7\right] .$ A generalization of the grand
Lebesgue spaces are the spaces $L^{p),\theta }\left( \Omega \right) $, $%
\theta \geq 0$, defined by the norm%
\begin{equation*}
\left\Vert f\right\Vert _{p),\theta ,\Omega }=\underset{0<\varepsilon \leq
p-1}{\sup }\varepsilon ^{\frac{\theta }{p-\varepsilon }}\left( \underset{%
\Omega }{\dint }\left\vert f\right\vert ^{p-\varepsilon }d\mu \right) ^{%
\frac{1}{p-\varepsilon }}=\underset{0<\varepsilon \leq p-1}{\sup }%
\varepsilon ^{\frac{\theta }{p-\varepsilon }}\left\Vert f\right\Vert
_{p-\varepsilon };
\end{equation*}%
when $\theta =0$ the space $L^{p),0}\left( \Omega \right) $ reduces to
Lebesgue space $L^{p}\left( \Omega \right) $ and when $\theta =1$ the space $%
L^{p),1}\left( \Omega \right) $ reduces to grand Lebesgue space $%
L^{p)}\left( \Omega \right) $, (see, $\left[ 3\right] ).$ For $0<\varepsilon
\leq p-1$, 
\begin{equation*}
L^{p}\left( \Omega \right) \subset L^{p),\theta }\left( \Omega \right)
\subset L^{p-\varepsilon }\left( \Omega \right)
\end{equation*}%
hold. It is known that the subspace $C_{c}^{\infty }\left( \Omega \right) $
is not dense in $L^{p),\theta }\left( \Omega \right) $, where $C_{c}^{\infty
}\left( \Omega \right) $ is the space of infinitely differentiable complex
valued functions defined on $\Omega $ with compact support. Its closure of
consists of functions $f\in L^{p}\left( \Omega \right) $ such that 
\begin{equation*}
\underset{\varepsilon \rightarrow 0}{\lim }\varepsilon ^{\frac{\theta }{%
p-\varepsilon }}\left\Vert f\right\Vert _{p-\varepsilon }=0\text{, }\left[ 3%
\right] .
\end{equation*}%
For some properties and applications of $L^{p}\left( \Omega \right) $, we
refer to papers $\left[ 2\right] $, $\left[ 4\right] $, $\left[ 5\right] $
and $\left[ 6\right] .$

Let $p^{\prime }=\frac{p}{p-1}$, $1<p<\infty $. First consider an auxiliary
space namely $L^{(p^{\prime },\theta }\left( \Omega \right) $, $\theta \geq
0 $, defined by%
\begin{equation*}
\left\Vert g\right\Vert _{(p^{\prime },\theta }=\underset{g=\underset{k=1}{%
\overset{\infty }{\dsum }}g_{k}}{\inf }\left\{ \underset{k=1}{\overset{%
\infty }{\dsum }}\underset{0<\varepsilon <p-1}{\inf }\varepsilon ^{-\frac{%
\theta }{p-\varepsilon }}\left( \underset{\Omega }{\doint }\left\vert
g_{k}\left( x\right) \right\vert ^{\left( p-\varepsilon \right) ^{\prime
}}d\left( x\right) \right) ^{\frac{1}{\left( p-\varepsilon \right) ^{\prime }%
}}\right\}
\end{equation*}%
where the functions $g_{k}$, $k\in 
\mathbb{N}
$, being in $\mathit{M}_{0},$ the set of all real valued measurable
functions, finite a.e. in $\Omega $. After this definition the generalized
small Lebesgue spaces have been defined by 
\begin{equation*}
L^{p)^{\prime },\theta }\left( \Omega \right) =\left\{ g\in \mathit{M}_{0}%
\text{: }\left\Vert g\right\Vert _{p)^{\prime },\theta }<\infty \right\} ,
\end{equation*}%
where%
\begin{equation*}
\left\Vert g\right\Vert _{p)^{\prime },\theta }=\underset{%
\begin{array}{c}
0\leq \psi \leq \left\vert g\right\vert \\ 
\psi \in L^{(p^{\prime },\theta }\left( \Omega \right)%
\end{array}%
}{\sup }\left\Vert \psi \right\Vert _{(p^{\prime },\theta }.
\end{equation*}%
For $\theta =0$ it is $\left\Vert f\right\Vert _{(p^{\prime },0}=\left\Vert
f\right\Vert _{p)^{\prime },\theta }$, $\left[ 2\right] $, $\left[ 6\right]
. $

Let $G$ be a locally compact abelian metric group with Haar measure $\lambda 
$ and\ let $\hat{G}$ be dual group with Haar measure $\mu $. The translation
and modulation operators are given by 
\begin{equation*}
T_{x}f\left( t\right) =f\left( t-x\right) \text{, }M_{\xi }f\left( t\right)
=\left\langle t,\xi \right\rangle f\left( t\right) \text{, }t,x\in G,\text{ }%
\xi \in \hat{G}.
\end{equation*}%
For a function $f\in L^{1}\left( G\right) $, the function $\hat{f}$ defined
on $\hat{G}$ by 
\begin{equation*}
\hat{f}\left( \gamma \right) =\underset{G}{\int }f\left( x\right)
\left\langle \gamma ,-x\right\rangle d\lambda \left( x\right) \text{, }%
\gamma \in \hat{G}
\end{equation*}%
is called the Fourier transform of $f$, $\left[ 12\right] .$

\section{Main Results}

Let $G$ be a locally compact abelian metric group and $\hat{G}$ its dual
with Haar measures $\lambda $ and $\mu $ respectively. Before give the
definition of bilinear multiplier on $G$ of type $\left[ (p_{1}^{\prime
};(p_{2}^{\prime };(p_{3}^{\prime }\right] _{\theta },$ we remember that if $%
\lambda \left( G\right) $ is finite, then $G$ is compact. Thus the $\hat{G}$
dual of $G$ (Pontryagin dual) is a discrete group, and the dual measure on
this group is the counting measure. Also since $G$ is compact abelian metric
group is $\hat{G}$ is countable, $[12]$.

\begin{definition}
Let $G$ be a locally compact abelian metric group with Haar measure $\lambda 
$ and $\hat{G}$ its dual with Haar measure $\mu ,$ and $\lambda \left(
G\right) $ is finite. Assume that $~1<p_{i}<\infty $, $p_{i}^{\prime }=\frac{%
p_{i}}{p_{i}-1}$, $\left( i=1,2,3\right) $ and $\theta \geq 0$. We also
assume that $m\left( s,t\right) $ is a bounded measurable function on $\hat{G%
}\times \hat{G}$. Consider the bilinear operator $B_{m}$ associated with the
symbol $m$ 
\begin{equation*}
B_{m}\left( f,g\right) \left( x\right) =\dsum\limits_{s\in \hat{G}%
}\dsum\limits_{t\in \hat{G}}\hat{f}\left( s\right) \hat{g}\left( t\right)
m\left( s,t\right) \left\langle s+t,x\right\rangle ,
\end{equation*}%
defined for functions \ $f$, $g\in C^{\infty }\left( G\right) $. $m$ is said
to be a bilinear multiplier on $G$ of type $\left[ (p_{1}^{\prime
};(p_{2}^{\prime };(p_{3}^{\prime }\right] _{\theta }$, if there exists $C>0$
such that 
\begin{equation*}
\left\Vert B_{m}\left( f,g\right) \right\Vert _{(p_{3}^{\prime },\theta
}\leq C\left\Vert f\right\Vert _{(p_{1}^{\prime },\theta }\left\Vert
g\right\Vert _{(p_{2}^{\prime },\theta }
\end{equation*}%
for all \ \ $f$, $g\in C^{\infty }\left( G\right) $. That means $B_{m}$
extends to a bounded bilinear operator from $L^{(p_{1}^{\prime },\theta
}\left( G\right) \times L^{(p_{2}^{\prime },\theta }\left( G\right) $ into $%
L^{(p_{3}^{\prime },\theta }\left( G\right) $. We denote by \ $BM_{\theta }%
\left[ (p_{1}^{\prime };(p_{2}^{\prime };(p_{3}^{\prime }\right] $ the space
of all bilinear multipliers of type $\left[ (p_{1}^{\prime };(p_{2}^{\prime
};(p_{3}^{\prime }\right] _{\theta }$ and%
\begin{equation*}
\left\Vert m\right\Vert _{\left[ (p_{1}^{\prime };(p_{2}^{\prime
};(p_{3}^{\prime }\right] _{\theta }}=\left\Vert B_{m}\right\Vert .
\end{equation*}
\end{definition}

\begin{lemma}
(Generalized H\"{o}lder inequality for generalized small Lebesgue spaces)
\end{lemma}

Let $\frac{1}{p^{\prime }}=\frac{1}{p_{1}^{\prime }}+$ $\frac{1}{%
p_{2}^{\prime }}$ and $r.r^{\prime }<p^{\prime }+1.$ If\ $\ f\in
L^{(p_{1}^{\prime },\theta }$ $\left( G\right) $ and $g\in L^{(p_{2}^{\prime
},\theta }\left( G\right) $, then $fg\in L^{(r^{\prime },\theta }$ $\left(
G\right) $. Furthermore

\begin{equation*}
\left\Vert fg\right\Vert _{(r^{\prime },\theta }\leq C\left\Vert
f\right\Vert _{(p_{1}^{\prime },\theta }\left\Vert g\right\Vert
_{(p_{2}^{\prime },\theta }
\end{equation*}%
for some $C>0.$

\begin{proof}
Take any $f\in C^{\infty }\left( G\right) \subset L^{(p_{1}^{\prime },\theta
}\left( G\right) $ and $g\in C^{\infty }\left( G\right) \subset
L^{(p_{2}^{\prime },\theta }\left( G\right) $. Let $\frac{1}{p^{\prime }}=%
\frac{1}{p_{1}^{\prime }}+\frac{1}{p_{2}^{\prime }}$ and $r.r^{\prime
}<p^{\prime }+1$. Since $1=\frac{1}{r}+$ $\frac{1}{r^{\prime }}$, we have $%
r+r^{\prime }=r.r^{\prime }$. Then using the assumption $r.r^{\prime
}<p^{\prime }+1,$ we write $r+r^{\prime }<p^{\prime }$ $+1$ and so $%
r+r^{\prime }-1<p^{\prime }$. For a fixed $0<\varepsilon \leq r-1$, we have $%
r^{^{\prime }}+\varepsilon \leq r^{\prime }+r-1<p^{\prime }.$ Therefore
since $\mu \left( G\right) <\infty ,$we obtain $L^{p^{\prime }}\left(
G\right) \subset L^{r^{\prime }+\varepsilon }\left( G\right) $. Also we know
that $L^{r^{\prime }+\varepsilon }\left( G\right) \subset L^{(r^{\prime
},\theta }$ $\left( G\right) ,$ $[2]$. Then we have the inclusion $%
L^{p^{\prime }}\left( G\right) \subset L^{r^{\prime }+\varepsilon }\left(
G\right) \subset L^{(r^{\prime },\theta }$ $\left( G\right) $. That means,
there exists $C_{1}>0$ such that%
\begin{equation}
\left\Vert fg\right\Vert _{(r^{\prime },\theta }\leq C_{1}\left\Vert
fg\right\Vert _{p^{\prime }}.  \tag{2.1}
\end{equation}%
If we apply the H\"{o}lder inequality to the right side of $\left(
2.1\right) $, there exists $C_{2}>0$ such that 
\begin{equation}
\left\Vert fg\right\Vert _{p^{\prime }}\leq C_{2}\left\Vert f\right\Vert
_{p_{1}^{\prime }}\left\Vert g\right\Vert _{p_{2}^{\prime }}.  \tag{2.2}
\end{equation}%
On the other hand, since $L^{(p_{1}^{\prime },\theta }\left( G\right)
\subset L^{p_{1}^{\prime }}\left( G\right) $ and $L^{(p_{2}^{\prime },\theta
}\left( G\right) \subset L^{p_{2}^{\prime }}\left( G\right) ,$ $[2]$, we
have 
\begin{equation}
\left\Vert f\right\Vert _{p_{1}^{\prime }}\leq C_{3}\left\Vert f\right\Vert
_{(p_{1}^{\prime },\theta }  \tag{2.3}
\end{equation}%
and%
\begin{equation}
\left\Vert g\right\Vert _{p_{2}^{\prime }}\leq C_{4}\left\Vert g\right\Vert
_{(p_{2}^{\prime },\theta }  \tag{2.4}
\end{equation}%
for some $C_{3}$,\ $C_{4}>0.$ Combining the inequalities (2.1), (2.2), (2.3)
and (2.4), we obtain 
\begin{equation}
\left\Vert fg\right\Vert _{(r^{\prime },\theta }\leq C\left\Vert
f\right\Vert _{(p_{1}^{\prime },\theta }\left\Vert g\right\Vert
_{(p_{2}^{\prime },\theta }  \tag{2.5}
\end{equation}%
where $\ C=C_{1}C_{2}C_{3}C_{4}.$ Now define the bilinear mapping $F\left(
\left( f,g\right) \right) =fg,$ from $\left( C^{\infty }\times C^{\infty
}\right) \left( G\right) $ to $L^{(r^{\prime },\theta }\left( G\right) .$ By 
$\left( 2.5\right) ,$ it is continuous. Since $\left( C^{\infty }\times
C^{\infty }\right) \left( G\right) $ is dense in $L^{(p_{1}^{\prime },\theta
}$ $\left( G\right) \times L^{(p_{2}^{\prime },\theta }\left( G\right) ,$
then there exists has a unique continuous bilinear extension of $F$ denoted $%
F^{\sim }$ from $L^{(p_{1}^{\prime },\theta }$ $\left( G\right) \times
L^{(p_{2}^{\prime },\theta }\left( G\right) $ to $L^{(r^{\prime },\theta
}\left( G\right) .$ Furthermore, the norm of $F^{\sim }$ is equal to the
norm of $F.$ Therefore, for all $f\in L^{(p_{1}^{\prime },\theta }$ $\left(
G\right) $ and $g\in L^{(p_{2}^{\prime },\theta }\left( G\right) $, the
inequality%
\begin{equation*}
\left\Vert fg\right\Vert _{(r^{\prime },\theta }=\left\Vert F\left( \left(
f,g\right) \right) \right\Vert _{(r^{\prime },\theta }\leq C\left\Vert
f\right\Vert _{(p_{1}^{\prime },\theta }\left\Vert g\right\Vert
_{(p_{2}^{\prime },\theta }
\end{equation*}%
is achieved.
\end{proof}

\begin{example}
Let $f\in L^{(9,\theta }$ $\left( G\right) $ and $g\in L^{(10,\theta }\left(
G\right) $. Since $p_{1}^{\prime }=9$ and $p_{2}^{\prime }=10,$ then $\frac{1%
}{p^{\prime }}=\frac{1}{p_{1}^{\prime }}+\frac{1}{p_{2}^{\prime }}=\frac{1}{9%
}+$ $\frac{1}{10},$ and so $p^{\prime }=\frac{90}{19}$. Also let $r=3$ and $%
r^{\prime }=\frac{3}{2}$. Then $\frac{1}{r}+\frac{1}{r^{\prime }}=1$. Hence
we have $r.r^{\prime }=3.\frac{3}{2}=\frac{9}{2}<\frac{90}{19}+1=p^{\prime
}+1$. Therefore from the Lemma 1, we obtain that $fg\in L^{(\frac{3}{2}%
,\theta }$ $\left( G\right) $ and 
\begin{equation*}
\left\Vert fg\right\Vert _{(\frac{3}{2},\theta }\leq C\left\Vert
f\right\Vert _{(9,\theta }\left\Vert g\right\Vert _{(10,\theta }
\end{equation*}%
for some $C>0$.
\end{example}

\begin{theorem}
Let$~1<p_{i}<\infty $, $p_{i}^{\prime }=\frac{p_{i}}{p_{i}-1}$, $\left(
i=1,2,3\right) $ and $\theta >0$. Then $m\in BM_{\theta }\left[
(p_{1}^{\prime };(p_{2}^{\prime };(p_{3}^{\prime }\right] $ if and only if
there exists $C>0$ such that 
\begin{equation*}
\left\vert \dsum\limits_{s\in \hat{G}}\dsum\limits_{t\in \hat{G}}\hat{f}%
\left( s\right) \hat{g}\left( t\right) \hat{h}\left( s+t\right) m\left(
s,t\right) \right\vert \leq C\left\Vert f\right\Vert _{(p_{1}^{\prime
},\theta }\left\Vert g\right\Vert _{(p_{2}^{\prime },\theta }\left\Vert
h\right\Vert _{p_{3}),\theta }
\end{equation*}%
for all \ $f\in L^{(p_{1}^{\prime },\theta }\left( G\right) $, $g\in
L^{(p_{2}^{\prime },\theta }\left( G\right) $ and $h\in $ $\left[ L^{p_{3}}%
\right] _{p_{3}),\theta }$ , where $\left[ L^{p_{3}}\right] _{p_{3}),\theta
} $ is the closure of $C^{\infty }\left( G\right) $ in $L^{p_{3}),\theta
}\left( G\right) $.
\end{theorem}

\begin{proof}
Assume that $m\in BM_{\theta }\left[ (p_{1}^{\prime };(p_{2}^{\prime
};(p_{3}^{\prime }\right] $. Let \ $f$, $g\in C^{\infty }\left( G\right)
\subset $ $L^{1}\left( G\right) $ and $h\in $ $C^{\infty }\left( G\right)
\cap L^{p_{3}),\theta }\left( G\right) $ . Since $\mu \left( G\right)
<\infty $, we have $h\in $ $L^{1}\left( G\right) $. So, we write%
\begin{equation*}
\left\vert \dsum\limits_{s\in \hat{G}}\dsum\limits_{t\in \hat{G}}\hat{f}%
\left( s\right) \hat{g}\left( t\right) \hat{h}\left( s+t\right) m\left(
s,t\right) \right\vert
\end{equation*}%
\begin{equation*}
=\left\vert \dsum\limits_{s\in \hat{G}}\dsum\limits_{t\in \hat{G}}\hat{f}%
\left( s\right) \hat{g}\left( t\right) \left\{ \underset{G}{\int }h\left(
y\right) \left\langle s+t,-y\right\rangle d\lambda \left( y\right) \right\}
m\left( s,t\right) \right\vert
\end{equation*}%
\begin{equation*}
=\left\vert \underset{G}{\int }h\left( y\right) B_{m}\left( f,g\right)
\left( -y\right) d\lambda \left( y\right) \right\vert =\left\vert \underset{G%
}{\int }h\left( y\right) \tilde{B}_{m}\left( f,g\right) \left( y\right)
d\lambda \left( y\right) \right\vert
\end{equation*}%
\begin{equation}
\leq \underset{G}{\int }\left\vert h\left( y\right) \right\vert \left\vert 
\tilde{B}_{m}\left( f,g\right) \left( y\right) \right\vert d\lambda \left(
y\right)  \tag{2.6}
\end{equation}%
where $\lambda $ is Haar measure on $G$ and $\tilde{B}_{m}\left( f,g\right)
\left( y\right) =B_{m}\left( f,g\right) \left( -y\right) $. From the
assumption $m\in BM_{\theta }\left[ (p_{1}^{\prime };(p_{2}^{\prime
};(p_{3}^{\prime }\right] $, then $B_{m}\left( f,g\right) \in $ $%
L^{(p_{3}^{\prime },\theta }\left( G\right) $. Since $G$ is group, we obtain 
$\tilde{B}_{m}\left( f,g\right) \in $ $L^{(p_{3}^{\prime },\theta }\left(
G\right) $. By using the H\"{o}lder inequality for generalized small
Lebesgue spaces and the inequality (2.6), we write%
\begin{equation*}
\left\vert \dsum\limits_{s\in \hat{G}}\dsum\limits_{t\in \hat{G}}\hat{f}%
\left( s\right) \hat{g}\left( t\right) \hat{h}\left( s+t\right) m\left(
s,t\right) \right\vert \leq \underset{G}{\int }\left\vert h\left( y\right)
\right\vert \left\vert \tilde{B}_{m}\left( f,g\right) \left( y\right)
\right\vert d\lambda \left( y\right)
\end{equation*}%
\begin{equation}
\leq \left\Vert \tilde{B}_{m}\left( f,g\right) \right\Vert _{(p_{3}^{\prime
},\theta }\left\Vert h\right\Vert _{p_{3}),\theta }=\left\Vert B_{m}\left(
f,g\right) \right\Vert _{(p_{3}^{\prime },\theta }\left\Vert h\right\Vert
_{p_{3}),\theta }  \tag{2.7}
\end{equation}%
Also since $m\in BM_{\theta }\left[ (p_{1}^{\prime };(p_{2}^{\prime
};(p_{3}^{\prime }\right] $, there exists $C>0$ such that 
\begin{equation}
\left\Vert B_{m}\left( f,g\right) \right\Vert _{(p_{3}^{\prime },\theta
}\leq C\left\Vert f\right\Vert _{(p_{1}^{\prime },\theta }\left\Vert
g\right\Vert _{(p_{2}^{\prime },\theta }  \tag{2.8}
\end{equation}%
Combining (2.7) and (2.8), we find%
\begin{equation*}
\left\vert \dsum\limits_{s\in \hat{G}}\dsum\limits_{t\in \hat{G}}\hat{f}%
\left( s\right) \hat{g}\left( t\right) \hat{h}\left( s+t\right) m\left(
s,t\right) \right\vert \leq C\left\Vert f\right\Vert _{(p_{1}^{\prime
},\theta }\left\Vert g\right\Vert _{(p_{2}^{\prime },\theta }\left\Vert
h\right\Vert _{p_{3}),\theta }\text{.}
\end{equation*}%
\qquad For the proof of converse, assume that there exists a constant $C>0$
such that%
\begin{equation*}
\left\vert \dsum\limits_{s\in \hat{G}}\dsum\limits_{t\in \hat{G}}\hat{f}%
\left( s\right) \hat{g}\left( t\right) \hat{h}\left( s+t\right) m\left(
s,t\right) \right\vert \leq C\left\Vert f\right\Vert _{(p_{1}^{\prime
},\theta }\left\Vert g\right\Vert _{(p_{2}^{\prime },\theta }\left\Vert
h\right\Vert _{p_{3}),\theta }
\end{equation*}%
for all \ $f$, $g\in C^{\infty }\left( G\right) $ and , $h\in $ $C^{\infty
}\left( G\right) \cap L^{p_{3}),\theta }\left( G\right) $. From the
assumption and (2.7), we write 
\begin{equation}
\left\vert \underset{G}{\int }h\left( y\right) \tilde{B}_{m}\left(
f,g\right) \left( y\right) d\lambda \left( y\right) \right\vert \leq
C\left\Vert f\right\Vert _{(p_{1}^{\prime },\theta }\left\Vert g\right\Vert
_{(p_{2}^{\prime },\theta }\left\Vert h\right\Vert _{p_{3}),\theta }\text{.}
\tag{2.9}
\end{equation}%
Define a function $\ell $ from $C^{\infty }\left( G\right) \cap
L^{p_{3}),\theta }\left( G\right) $ to $%
\mathbb{C}
$ such that%
\begin{equation*}
\ell \left( h\right) =\underset{G}{\int }h\left( y\right) \tilde{B}%
_{m}\left( f,g\right) \left( y\right) d\lambda \left( y\right) \text{.}
\end{equation*}%
It is clear that the function $\ell $ is linear and bounded by (2.9). \
Since $\overline{C^{\infty }\left( G\right) \cap L^{p_{3}),\theta }\left(
G\right) }$ = $\left[ L^{p_{3}}\right] _{p_{3}),\theta }$ . Then $\ell \in
\left( \left[ L^{p_{3}}\right] _{p_{3}),\theta }\right) ^{\ast
}=L^{p_{3})^{\prime },\theta }\left( G\right) \simeq L^{(p_{3}{}^{\prime
},\theta }\left( G\right) $ and by (2.9), we have%
\begin{equation*}
\left\Vert B_{m}\left( f,g\right) \right\Vert _{(p_{3}^{\prime },\theta
}=\left\Vert \tilde{B}_{m}\left( f,g\right) \right\Vert _{(p_{3}^{\prime
},\theta }=\left\Vert \ell \right\Vert =\underset{\left\Vert h\right\Vert
_{p_{3}),\theta }\leq 1}{\sup }\frac{\left\vert l\left( h\right) \right\vert 
}{\left\Vert h\right\Vert _{p_{3}),\theta }}
\end{equation*}%
\begin{equation*}
\leq \underset{\left\Vert h\right\Vert _{p_{3}),\theta }\leq 1}{\sup }\frac{%
C\left\Vert f\right\Vert _{(p_{1}^{\prime },\theta }\left\Vert g\right\Vert
_{(p_{2}^{\prime },\theta }\left\Vert h\right\Vert _{p_{3}^{\prime },\omega
_{3}^{-1}}}{\left\Vert h\right\Vert _{p_{3}),\theta }}\leq C\left\Vert
f\right\Vert _{(p_{1}^{\prime },\theta }\left\Vert g\right\Vert
_{(p_{2}^{\prime },\theta }\text{.}
\end{equation*}%
Hence, we obtain $m\in BM_{\theta }\left[ (p_{1}^{\prime };(p_{2}^{\prime
};(p_{3}^{\prime }\right] $
\end{proof}

\begin{theorem}
Let $\ m\in BM_{\theta }\left[ (p_{1}^{\prime };(p_{2}^{\prime
};(p_{3}^{\prime }\right] $. Then

\textbf{a)} $T_{\left( s_{0},t_{0}\right) }m\in BM_{\theta }\left[
(p_{1}^{\prime };(p_{2}^{\prime };(p_{3}^{\prime }\right] $ for each $\left(
s_{0},t_{0}\right) \in \hat{G}\times \hat{G}$ and 
\begin{equation*}
\left\Vert T_{\left( s_{0},t_{0}\right) }m\right\Vert _{\left[
(p_{1}^{\prime };(p_{2}^{\prime };(p_{3}^{\prime }\right] _{\theta
}}=\left\Vert m\right\Vert _{\left[ (p_{1}^{\prime };(p_{2}^{\prime
};(p_{3}^{\prime }\right] _{\theta }}\text{.}
\end{equation*}

\textbf{b)}$M_{t_{0}}^{2}M_{s_{0}}^{1}m\in BM_{\theta }\left[ (p_{1}^{\prime
};(p_{2}^{\prime };(p_{3}^{\prime }\right] $ for each $\left(
s_{0},t_{0}\right) \in G\times G$ and%
\begin{equation*}
\left\Vert M_{t_{0}}^{2}M_{s_{0}}^{1}m\right\Vert _{\left[ (p_{1}^{\prime
};(p_{2}^{\prime };(p_{3}^{\prime }\right] _{\theta }}=\left\Vert
m\right\Vert _{\left[ (p_{1}^{\prime };(p_{2}^{\prime };(p_{3}^{\prime }%
\right] _{\theta }}
\end{equation*}%
where $M_{s_{0}}^{1}m\left( s,t\right) =\left\langle s,s_{0}\right\rangle
m\left( s,t\right) $ and $M_{t_{0}}^{2}m\left( s,t\right) =\left\langle
t,t_{0}\right\rangle m\left( s,t\right) .$
\end{theorem}

\begin{proof}
\textbf{a) }Let \ $f$, $g\in C^{\infty }\left( G\right) $. Let say $%
s-s_{0}=u $ and $t-t_{0}=v$. Then%
\begin{equation*}
B_{T_{\left( s_{0},t_{0}\right) }m}\left( f,g\right) \left( x\right)
=\dsum\limits_{s\in \hat{G}}\dsum\limits_{t\in \hat{G}}\hat{f}\left(
s\right) \hat{g}\left( t\right) T_{\left( s_{0},t_{0}\right) }m\left(
s,t\right) \left\langle s+t,x\right\rangle
\end{equation*}%
\begin{equation*}
=\dsum\limits_{u\in \hat{G}}\dsum\limits_{v\in \hat{G}}\hat{f}\left(
u+s_{0}\right) \hat{g}\left( v+t_{0}\right) m\left( u,v\right) \left\langle
u+s_{0}+v+t_{0},x\right\rangle
\end{equation*}

\begin{equation*}
=\dsum\limits_{u\in \hat{G}}\dsum\limits_{v\in \hat{G}}\hat{f}\left(
u+s_{0}\right) \hat{g}\left( v+t_{0}\right) m\left( u,v\right) \left\langle
(s_{0}+t_{0})+(u+v),x\right\rangle
\end{equation*}

\begin{equation}
=\dsum\limits_{u\in \hat{G}}\dsum\limits_{v\in \hat{G}}T_{-s_{0}}\hat{f}%
\left( u\right) T_{-t_{0}}\hat{g}\left( v\right) m\left( u,v\right)
\left\langle s_{0}+t_{0},x\right\rangle \left\langle u+v,x\right\rangle 
\tag{2.10}
\end{equation}

It is known that

\begin{equation*}
\left( M_{-s_{0}}f\right) \symbol{94}\left( u\right) =\underset{G}{\int }%
M_{-s_{0}}f\left( y\right) \left\langle u,-y\right\rangle d\lambda \left(
y\right)
\end{equation*}%
\begin{equation*}
=\underset{G}{\int }\left\langle -s_{0},y\right\rangle f\left( y\right)
\left\langle u,-y\right\rangle d\lambda \left( y\right) =\underset{G}{\int }%
\left\langle s_{0},-y\right\rangle f\left( y\right) \left\langle
u,-y\right\rangle d\lambda \left( y\right)
\end{equation*}%
\begin{equation}
=\underset{G}{\int }f\left( y\right) \left\langle s_{0}+u,-y\right\rangle
d\lambda \left( y\right) =\hat{f}\left( s_{0}+u\right) =T_{-s_{0}}\hat{f}%
\left( u\right) .  \tag{2.11}
\end{equation}%
Similarly we obtain $\left( M_{-t_{0}}g\right) \symbol{94}=T_{-t_{0}}\hat{g}%
. $ Combining (2.10) and (2.11), we have

\begin{equation*}
B_{T_{\left( s_{0},t_{0}\right) }m}\left( f,g\right) \left( x\right)
=\dsum\limits_{u\in \hat{G}}\dsum\limits_{v\in \hat{G}}T_{-s_{0}}\hat{f}%
\left( u\right) T_{-t_{0}}\hat{g}\left( v\right) m\left( u,v\right)
\left\langle s_{0}+t_{0},x\right\rangle \left\langle s+t,x\right\rangle
\end{equation*}%
\begin{equation*}
=\left\langle s_{0}+t_{0},x\right\rangle \dsum\limits_{u\in \hat{G}%
}\dsum\limits_{v\in \hat{G}}\left( M_{-s_{0}}f\right) \symbol{94}\left(
u\right) \left( M_{-t_{0}}g\right) \symbol{94}\left( v\right) m\left(
u,v\right) \left\langle s+t,x\right\rangle
\end{equation*}%
\begin{equation}
=\left\langle s_{0}+t_{0},x\right\rangle B_{m}\left(
M_{-s_{0}}f,M_{-t_{0}}g\right) \left( x\right) \text{.}  \tag{2.12}
\end{equation}%
Using the assumption $m\in BM_{\theta }[(p_{1}^{\prime };(p_{2}^{\prime
};(p_{3}^{\prime }]$ and the equality (2.12) we have 
\begin{equation*}
\left\Vert B_{T_{\left( \xi _{0},\eta _{0}\right) }m}\left( f,g\right)
\right\Vert _{(p_{3}^{\prime },\theta }=\left\Vert \left\langle
s_{0}+t_{0},x\right\rangle B_{m}\left( M_{-s_{0}}f,M_{-t_{0}}g\right)
\right\Vert _{(p_{3}^{\prime },\theta }
\end{equation*}%
\begin{equation*}
=\left\Vert B_{m}\left( M_{-s_{0}}f,M_{-t_{0}}g\right) \right\Vert
_{(p_{3}^{\prime },\theta }
\end{equation*}%
\begin{equation*}
\leq C\left\Vert M_{-s_{0}}f\right\Vert _{(p_{1}^{\prime },\theta
}\left\Vert M_{-t_{0}}g\right\Vert _{(p_{2}^{\prime },\theta }
\end{equation*}%
\begin{equation*}
=C\left\Vert \left\langle -s_{0},.\right\rangle f\left( .\right) \right\Vert
_{(p_{1}^{\prime },\theta }\left\Vert \left\langle -t_{0},.\right\rangle
g\left( .\right) \right\Vert _{(p_{2}^{\prime },\theta }
\end{equation*}%
\begin{equation*}
=C\left\Vert f\right\Vert _{(p_{1}^{\prime },\theta }\left\Vert g\right\Vert
_{(p_{2}^{\prime },\theta }
\end{equation*}%
for some $C>0$. Thus $T$ $_{_{\left( s_{0},t_{0}\right) }m}\in BM_{\theta
}[(p_{1}^{\prime };(p_{2}^{\prime };(p_{3}^{\prime }]$. Then by definition
1, 
\begin{equation*}
\left\Vert T_{\left( s_{0},t_{0}\right) }m\right\Vert _{\left[
(p_{1}^{\prime };(p_{2}^{\prime };(p_{3}^{\prime }\right] _{\theta
}}=\left\Vert B_{T_{\left( s_{0},t_{0}\right) }m}\right\Vert .
\end{equation*}

This implies

\begin{equation*}
\left\Vert T_{\left( s_{0},t_{0}\right) }m\right\Vert _{\left[
(p_{1}^{\prime };(p_{2}^{\prime };(p_{3}^{\prime }\right] _{\theta
}}=\left\Vert B_{T_{\left( s_{0},t_{0}\right) }m}\right\Vert
\end{equation*}%
\begin{equation*}
=\sup \left\{ \frac{\left\Vert B_{T_{\left( \xi _{0},\eta _{0}\right)
}m}\left( f,g\right) \right\Vert _{(p_{3}^{\prime },\theta }}{\left\Vert
f\right\Vert _{(p_{1}^{\prime },\theta }\left\Vert g\right\Vert
_{(p_{2}^{\prime },\theta }}:\left\Vert f\right\Vert _{(p_{1}^{\prime
},\theta }\leq 1\text{, }\left\Vert g\right\Vert _{(p_{2}^{\prime },\theta
}\leq 1\right\}
\end{equation*}%
\begin{equation*}
=\sup \left\{ \frac{\left\Vert B_{m}\left( M_{-s_{0}}f,M_{-t_{0}}g\right)
\right\Vert _{(p_{3}^{\prime },\theta }}{\left\Vert M_{-s_{0}}f\right\Vert
_{(p_{1}^{\prime },\theta }\left\Vert M_{-t_{0}}g\right\Vert
_{(p_{2}^{\prime },\theta }}:\left\Vert M_{-s_{0}}f\right\Vert
_{(p_{1}^{\prime },\theta }\leq 1\text{, }\left\Vert M_{-t_{0}}g\right\Vert
_{(p_{2}^{\prime },\theta }\leq 1\right\}
\end{equation*}%
\begin{equation*}
=\left\Vert B_{m}\right\Vert =\left\Vert m\right\Vert _{\left[
(p_{1}^{\prime };(p_{2}^{\prime };(p_{3}^{\prime }\right] _{\theta }}\text{.}
\end{equation*}

\textbf{b) }Let $f$, $g\in C^{\infty }\left( G\right) $. We have 
\begin{equation*}
B_{M_{t_{0}}^{2}M_{s_{0}}^{1}m}\left( f,g\right) \left( x\right)
=\dsum\limits_{s\in \hat{G}}\dsum\limits_{t\in \hat{G}}\hat{f}\left(
s\right) \hat{g}\left( t\right) M_{t_{0}}^{2}M_{s_{0}}^{1}\left( m\left(
s,t\right) \right) \left\langle s+t,x\right\rangle
\end{equation*}%
\begin{equation*}
=\dsum\limits_{s\in \hat{G}}\dsum\limits_{t\in \hat{G}}\hat{f}\left(
s\right) \hat{g}\left( t\right) M_{t_{0}}^{2}\left( \left\langle
s,s_{0}\right\rangle m\left( s,t\right) \right) \left\langle
s+t,x\right\rangle
\end{equation*}%
\begin{equation*}
=\dsum\limits_{s\in \hat{G}}\dsum\limits_{t\in \hat{G}}\hat{f}\left(
s\right) \hat{g}\left( t\right) \left\langle t,t_{0}\right\rangle
\left\langle s,s_{0}\right\rangle m\left( s,t\right) \left\langle
s+t,x\right\rangle
\end{equation*}%
\begin{equation*}
=\dsum\limits_{s\in \hat{G}}\dsum\limits_{t\in \hat{G}}\left\langle
s,s_{0}\right\rangle \hat{f}\left( s\right) \left\langle
t,t_{0}\right\rangle \hat{g}\left( t\right) m\left( s,t\right) \left\langle
s+t,x\right\rangle
\end{equation*}%
\begin{equation}
=\dsum\limits_{s\in \hat{G}}\dsum\limits_{t\in \hat{G}}M_{s_{0}}\hat{f}%
\left( s\right) M_{t_{0}}\hat{g}\left( t\right) m\left( s,t\right)
\left\langle s+t,x\right\rangle  \tag{2.13}
\end{equation}%
If we put $y+s_{0}=u$, we have 
\begin{equation*}
\left( T_{-s_{0}}f\right) \symbol{94}\left( s\right) =\underset{G}{\int }%
T_{-s_{0}}f\left( y\right) \left\langle s,-y\right\rangle d\lambda \left(
y\right)
\end{equation*}%
\begin{equation*}
=\underset{G}{\int }f\left( y+s_{0}\right) \left\langle s,-y\right\rangle
d\lambda \left( y\right) =\underset{G}{\int }f\left( u\right) \left\langle
s,-u+s_{0}\right\rangle d\lambda \left( u\right)
\end{equation*}%
\begin{equation*}
=\underset{G}{\int }f\left( u\right) \left\langle s,-u\right\rangle
\left\langle s,s_{0}\right\rangle d\lambda \left( u\right) =\left\langle
s,s_{0}\right\rangle \underset{G}{\int }f\left( u\right) \left\langle
s,-u\right\rangle d\lambda \left( u\right)
\end{equation*}%
\begin{equation*}
=\left\langle s,s_{0}\right\rangle \hat{f}(s)=M_{s_{0}}\hat{f}\left(
s\right) .
\end{equation*}%
Similarly $\left( T_{-t_{0}}g\right) \symbol{94}=M_{t_{0}}\hat{g}$ . Then
from (2.13)%
\begin{equation*}
B_{M_{t_{0}}^{2}M_{s_{0}}^{1}m}\left( f,g\right) \left( x\right)
=\dsum\limits_{s\in \hat{G}}\dsum\limits_{t\in \hat{G}}\left(
T_{-s_{0}}f\right) \symbol{94}\left( s\right) \left( T_{-t_{0}}g\right) 
\symbol{94}\left( t\right) m\left( s,t\right) \left\langle
s+t,x\right\rangle =B_{m}\left( T_{-s_{0}}f,T_{-t_{0}}g\right) \left(
x\right) \text{.}
\end{equation*}%
Since $m\in BM_{\theta }[(p_{1}^{\prime };(p_{2}^{\prime };(p_{3}^{\prime }]$%
, 
\begin{equation*}
\left\Vert B_{M_{t_{0}}^{2}M_{s_{0}}^{1}m}\left( f,g\right) \right\Vert
_{(p_{3}^{\prime },\theta }=\left\Vert B_{m}\left(
T_{-s_{0}}f,T_{-t_{0}}g\right) \right\Vert _{(p_{3}^{\prime },\theta }\leq
\left\Vert B_{m}\right\Vert \left\Vert T_{-s_{0}}f\right\Vert
_{(p_{1}^{\prime },\theta }\left\Vert T_{-t_{0}}g\right\Vert
_{(p_{2}^{\prime },\theta }
\end{equation*}%
\begin{equation}
\leq \left\Vert B_{m}\right\Vert \left\Vert f\right\Vert _{(p_{1}^{\prime
},\theta }\left\Vert g\right\Vert _{(p_{2}^{\prime },\theta }  \tag{2.14}
\end{equation}%
and so $M_{t_{0}}^{2}M_{s_{0}}^{1}m\in BM_{\theta }\left[ (p_{1}^{\prime
};(p_{2}^{\prime };(p_{3}^{\prime }\right] $. Finally (2.14), we achieve
that 
\begin{equation*}
\left\Vert M_{t_{0}}^{2}M_{s_{0}}^{1}m\right\Vert _{\left[ (p_{1}^{\prime
};(p_{2}^{\prime };(p_{3}^{\prime }\right] _{\theta }}=\sup \left\{ \frac{%
\left\Vert B_{M_{t_{0}}^{2}M_{s_{0}}^{1}m}\left( f,g\right) \right\Vert
_{(p_{3}^{\prime },\theta }}{\left\Vert f\right\Vert _{(p_{1}^{\prime
},\theta }\left\Vert g\right\Vert _{(p_{2}^{\prime },\theta }}:\left\Vert
f\right\Vert _{(p_{1}^{\prime },\theta }\leq 1\text{, }\left\Vert
g\right\Vert _{(p_{2}^{\prime },\theta }\leq 1\right\}
\end{equation*}%
\begin{equation*}
=\sup \left\{ \frac{\left\Vert B_{m}\left( T_{-s_{0}}f,T_{-t_{0}}g\right)
\right\Vert _{(p_{3}^{\prime },\theta }}{\left\Vert T_{-s_{0}}f\right\Vert
_{(p_{1}^{\prime },\theta }\left\Vert T_{-t_{0}}g\right\Vert
_{(p_{2}^{\prime },\theta }}:\left\Vert T_{-s_{0}}f\right\Vert
_{(p_{1}^{\prime },\theta }\leq 1\text{, }\left\Vert T_{-t_{0}}g\right\Vert
_{(p_{2}^{\prime },\theta }\leq 1\right\}
\end{equation*}

\begin{equation*}
=\left\Vert m\right\Vert _{\left[ (p_{1}^{\prime };(p_{2}^{\prime
};(p_{3}^{\prime }\right] _{\theta }}\text{.}
\end{equation*}
\end{proof}

Let $A$ be an automorphism of $G$. The $\lambda \circ A$ is a nontrival Haar
measure on $G$. For any Borel set $U\subseteq G$, the modules of $A$ is
defined by $\left\vert A\right\vert =$ $\lambda \left( AU\right) $ such that 
$d\lambda \left( Ax\right) =\left\vert A\right\vert d\lambda \left( x\right)
.$ Also the adjoint $A^{\ast }$of $A$ is an automorphism of $\hat{G}$. The
adjoint operator $A^{\ast }$ is defined by $\left\langle Ax,s\right\rangle
=\left\langle x,A^{\ast }s\right\rangle $ for $x\in G$ and $s\in \hat{G}$.
Furthermore the $\mu \circ A^{\ast }$ is a nontrival Haar measure on $\hat{G}
$ such that $d\mu \left( A^{\ast }s\right) =\left\vert A^{\ast }\right\vert
d\mu \left( s\right) $. It is known that $\left\vert A\right\vert
=\left\vert A^{\ast }\right\vert $, $\left( A^{\ast }\right) ^{-1}=\left(
A^{-1}\right) ^{\ast }$and $\left\vert A\right\vert ^{-1}=\left\vert
A^{-1}\right\vert $, $\left[ 1\right] .$

\begin{definition}
Let $A$ be an automorphism of $G$. The dilation operator $D_{A}^{p^{\prime
}} $ on $L^{(p^{\prime },\theta }\left( G\right) $ is defined by

\begin{equation*}
D_{A}^{p^{\prime }}f\left( x\right) =\left\vert A\right\vert ^{\frac{1}{%
p^{\prime }}}f\left( Ax\right) .
\end{equation*}
\end{definition}

\begin{lemma}
Let $A$ be an automorphism of G and $f\in L^{(p^{\prime },\theta }\left(
G\right) $. Then $D_{A}^{p^{\prime }}$ $f\in L^{(p^{\prime },\theta }\left(
G\right) $. Moreover%
\begin{equation*}
\left\Vert D_{A}^{p^{\prime }}f\right\Vert _{(p^{\prime },\theta
}=\left\vert A\right\vert ^{\frac{1}{p^{\prime }}}\left\Vert f\right\Vert
_{(p^{\prime },\theta }\leq \left\Vert f\right\Vert _{(p^{\prime },\theta }%
\text{, if \ }\left\vert A\right\vert <1
\end{equation*}%
\begin{equation*}
\left\Vert D_{A}^{p^{\prime }}f\right\Vert _{(p^{\prime },\theta
}=\left\Vert f\right\Vert _{(p^{\prime },\theta }\text{, \ \ \ \ \ \ \ \ \ \
\ \ \ \ \ \ \ \ \ \ \ \ \ if \ }\left\vert A\right\vert \geq 1.
\end{equation*}
\end{lemma}

\begin{proof}
Let $A$ be an automorphism of G and $f\in L^{(p^{\prime },\theta }\left(
G\right) $. If we say $Ax=u$ and use the equality $\left\vert A\right\vert
^{-1}=\left\vert A^{-1}\right\vert ,$ then 
\begin{equation*}
\left\Vert D_{A}^{p^{\prime }}f\right\Vert _{(p^{\prime },\theta }=\underset{%
D_{A}^{p^{\prime }}f=\underset{k=1}{\overset{\infty }{\dsum }}%
D_{A}^{p^{\prime }}f_{k}}{\inf }\left\{ \underset{k=1}{\overset{\infty }{%
\dsum }}\underset{0<\varepsilon <p-1}{\inf }\varepsilon ^{-\frac{\theta }{%
p-\varepsilon }}\left( \underset{G}{\doint }\left\vert D_{A}^{p^{\prime
}}f_{k}\left( x\right) \right\vert ^{\left( p-\varepsilon \right) ^{\prime
}}d\lambda \left( x\right) \right) ^{\frac{1}{\left( p-\varepsilon \right)
^{\prime }}}\right\}
\end{equation*}%
\begin{equation*}
=\underset{D_{A}^{p^{\prime }}f=\underset{k=1}{\overset{\infty }{\dsum }}%
D_{A}^{p^{\prime }}f_{k}}{\inf }\left\{ \underset{k=1}{\overset{\infty }{%
\dsum }}\underset{0<\varepsilon <p-1}{\inf }\varepsilon ^{-\frac{\theta }{%
p-\varepsilon }}\left( \underset{G}{\doint }\left\vert \left\vert
A\right\vert ^{\frac{1}{p^{\prime }}}f_{k}\left( Ax\right) \right\vert
^{\left( p-\varepsilon \right) ^{\prime }}d\lambda \left( x\right) \right) ^{%
\frac{1}{\left( p-\varepsilon \right) ^{\prime }}}\right\}
\end{equation*}

\begin{equation*}
=\underset{f=\underset{k=1}{\overset{\infty }{\dsum }}f_{k}}{\inf }\left\{ 
\underset{k=1}{\overset{\infty }{\dsum }}\underset{0<\varepsilon <p-1}{\inf }%
\varepsilon ^{-\frac{\theta }{p-\varepsilon }}\left( \underset{G}{\doint }%
\left\vert \left\vert A\right\vert ^{\frac{1}{p^{\prime }}}f_{k}\left(
u\right) \right\vert ^{\left( p-\varepsilon \right) ^{\prime }}d\lambda
\left( A^{-1}u\right) \right) ^{\frac{1}{\left( p-\varepsilon \right)
^{\prime }}}\right\}
\end{equation*}%
\begin{equation*}
=\underset{f=\underset{k=1}{\overset{\infty }{\dsum }}f_{k}}{\inf }\left\{ 
\underset{k=1}{\overset{\infty }{\dsum }}\underset{0<\varepsilon <p-1}{\inf }%
\varepsilon ^{-\frac{\theta }{p-\varepsilon }}\left( \underset{G}{\doint }%
\left\vert f_{k}\left( u\right) \right\vert ^{\left( p-\varepsilon \right)
^{\prime }}\left\vert A\right\vert ^{\frac{\left( p-\varepsilon \right)
^{\prime }}{p^{\prime }}}\left\vert A\right\vert ^{-1}d\lambda \left(
u\right) \right) ^{\frac{1}{\left( p-\varepsilon \right) ^{\prime }}}\right\}
\end{equation*}%
\begin{equation*}
=\underset{f=\underset{k=1}{\overset{\infty }{\dsum }}f_{k}}{\inf }\left\{ 
\underset{k=1}{\overset{\infty }{\dsum }}\underset{0<\varepsilon <p-1}{\inf }%
\varepsilon ^{-\frac{\theta }{p-\varepsilon }}\left\vert A\right\vert ^{%
\frac{1}{p^{\prime }}}\left\vert A\right\vert ^{-\frac{1}{\left(
p-\varepsilon \right) ^{\prime }}}\left( \underset{G}{\doint }\left\vert
f_{k}\left( u\right) \right\vert ^{\left( p-\varepsilon \right) ^{\prime
}}d\lambda \left( u\right) \right) ^{\frac{1}{\left( p-\varepsilon \right)
^{\prime }}}\right\} .
\end{equation*}%
\begin{equation*}
=\underset{f=\underset{k=1}{\overset{\infty }{\dsum }}f_{k}}{\inf }\left\{ 
\underset{k=1}{\overset{\infty }{\dsum }}\underset{0<\varepsilon <p-1}{\inf }%
\varepsilon ^{-\frac{\theta }{p-\varepsilon }}\left\vert A\right\vert ^{%
\frac{1}{p^{\prime }}}\left\vert A\right\vert ^{\frac{1}{\left(
p-\varepsilon \right) }-1}\left( \underset{G}{\doint }\left\vert f_{k}\left(
u\right) \right\vert ^{\left( p-\varepsilon \right) ^{\prime }}d\lambda
\left( u\right) \right) ^{\frac{1}{\left( p-\varepsilon \right) ^{\prime }}%
}\right\}
\end{equation*}%
\begin{equation*}
=\underset{f=\underset{k=1}{\overset{\infty }{\dsum }}f_{k}}{\inf }\left\{ 
\underset{k=1}{\overset{\infty }{\dsum }}\underset{0<\varepsilon <p-1}{\inf }%
\varepsilon ^{-\frac{\theta }{p-\varepsilon }}\left\vert A\right\vert ^{1-%
\frac{1}{p}}\left\vert A\right\vert ^{\frac{1}{\left( p-\varepsilon \right) }%
-1}\left( \underset{G}{\doint }\left\vert f_{k}\left( u\right) \right\vert
^{\left( p-\varepsilon \right) ^{\prime }}d\lambda \left( u\right) \right) ^{%
\frac{1}{\left( p-\varepsilon \right) ^{\prime }}}\right\}
\end{equation*}%
\begin{equation*}
=\underset{f=\underset{k=1}{\overset{\infty }{\dsum }}f_{k}}{\inf }\left\{ 
\underset{k=1}{\overset{\infty }{\dsum }}\underset{0<\varepsilon <p-1}{\inf }%
\varepsilon ^{-\frac{\theta }{p-\varepsilon }}\left\vert A\right\vert ^{%
\frac{1}{\left( p-\varepsilon \right) }-\frac{1}{p}}\left( \underset{G}{%
\doint }\left\vert f_{k}\left( u\right) \right\vert ^{\left( p-\varepsilon
\right) ^{\prime }}d\lambda \left( u\right) \right) ^{\frac{1}{\left(
p-\varepsilon \right) ^{\prime }}}\right\}
\end{equation*}%
\begin{equation}
=\underset{f=\underset{k=1}{\overset{\infty }{\dsum }}f_{k}}{\inf }\left\{ 
\underset{k=1}{\overset{\infty }{\dsum }}\underset{0<\varepsilon <p-1}{\inf }%
\left\vert A\right\vert ^{\frac{1}{\left( p-\varepsilon \right) }-\frac{1}{p}%
}\underset{0<\varepsilon <p-1}{\inf }\varepsilon ^{-\frac{\theta }{%
p-\varepsilon }}\left( \underset{G}{\doint }\left\vert f_{k}\left( u\right)
\right\vert ^{\left( p-\varepsilon \right) ^{\prime }}d\lambda \left(
u\right) \right) ^{\frac{1}{\left( p-\varepsilon \right) ^{\prime }}%
}\right\} .  \tag{2.15}
\end{equation}

\qquad \qquad \qquad

Assume that $\left\vert A\right\vert <1$ and $0<\varepsilon <p-1,$ since $%
\underset{0<\varepsilon <p-1}{\inf }\left\vert A\right\vert ^{\frac{1}{%
\left( p-\varepsilon \right) }-\frac{1}{p}}=\left\vert A\right\vert ^{1-%
\frac{1}{p}}=\left\vert A\right\vert ^{\frac{1}{p^{\prime }}}.$ Then by the
last inequality and (2.15), we have%
\begin{equation*}
\left\Vert D_{A}^{p^{\prime }}f\right\Vert _{(p^{\prime },\theta }=\underset{%
f=\underset{k=1}{\overset{\infty }{\dsum }}f_{k}}{\inf }\left\{ \underset{k=1%
}{\overset{\infty }{\dsum }}\underset{0<\varepsilon <p-1}{\inf }\left\vert
A\right\vert ^{\frac{1}{\left( p-\varepsilon \right) }-\frac{1}{p}}\underset{%
0<\varepsilon <p-1}{\inf }\varepsilon ^{-\frac{\theta }{p-\varepsilon }%
}\left( \underset{G}{\doint }\left\vert f_{k}\left( u\right) \right\vert
^{\left( p-\varepsilon \right) ^{\prime }}d\lambda \left( u\right) \right) ^{%
\frac{1}{\left( p-\varepsilon \right) ^{\prime }}}\right\}
\end{equation*}%
\begin{equation*}
=\underset{f=\underset{k=1}{\overset{\infty }{\dsum }}f_{k}}{\inf }\left\{ 
\underset{k=1}{\overset{\infty }{\dsum }}\left\vert A\right\vert ^{\frac{1}{%
p^{\prime }}}\underset{0<\varepsilon <p-1}{\inf }\varepsilon ^{-\frac{\theta 
}{p-\varepsilon }}\left( \underset{G}{\doint }\left\vert f_{k}\left(
u\right) \right\vert ^{\left( p-\varepsilon \right) ^{\prime }}d\lambda
\left( u\right) \right) ^{\frac{1}{\left( p-\varepsilon \right) ^{\prime }}%
}\right\}
\end{equation*}%
\begin{equation*}
=\left\vert A\right\vert ^{\frac{1}{p^{\prime }}}\underset{f=\underset{k=1}{%
\overset{\infty }{\dsum }}f_{k}}{\inf }\left\{ \underset{k=1}{\overset{%
\infty }{\dsum }}\underset{0<\varepsilon <p-1}{\inf }\varepsilon ^{-\frac{%
\theta }{p-\varepsilon }}\left( \underset{G}{\doint }\left\vert f_{k}\left(
u\right) \right\vert ^{\left( p-\varepsilon \right) ^{\prime }}d\lambda
\left( u\right) \right) ^{\frac{1}{\left( p-\varepsilon \right) ^{\prime }}%
}\right\} =\left\vert A\right\vert ^{\frac{1}{p^{\prime }}}\left\Vert
f\right\Vert _{(p^{\prime },\theta }.
\end{equation*}%
Thus $D_{A}^{p^{\prime }}f\in L^{(p^{\prime },\theta }\left( G\right) $.

Let $\left\vert A\right\vert \geq 1$, and let $0<\varepsilon <p-1.$ Since $%
\underset{0<\varepsilon <p-1}{\inf }\left\vert A\right\vert ^{\frac{1}{%
\left( p-\varepsilon \right) }-\frac{1}{p}}=1,$ by (2.15), we have%
\begin{equation*}
\left\Vert D_{A}^{p^{\prime }}f\right\Vert _{(p^{\prime },\theta }=\underset{%
f=\underset{k=1}{\overset{\infty }{\dsum }}f_{k}}{\inf }\left\{ \underset{k=1%
}{\overset{\infty }{\dsum }}\underset{0<\varepsilon <p-1}{\inf }\left\vert
A\right\vert ^{\frac{1}{\left( p-\varepsilon \right) }-\frac{1}{p}}\underset{%
0<\varepsilon <p-1}{\inf }\varepsilon ^{-\frac{\theta }{p-\varepsilon }%
}\left( \underset{G}{\doint }\left\vert f_{k}\left( u\right) \right\vert
^{\left( p-\varepsilon \right) ^{\prime }}d\lambda \left( u\right) \right) ^{%
\frac{1}{\left( p-\varepsilon \right) ^{\prime }}}\right\}
\end{equation*}%
\begin{equation*}
=\underset{f=\underset{k=1}{\overset{\infty }{\dsum }}f_{k}}{\inf }\left\{ 
\underset{k=1}{\overset{\infty }{\dsum }}\underset{0<\varepsilon <p-1}{\inf }%
\varepsilon ^{-\frac{\theta }{p-\varepsilon }}\left( \underset{G}{\doint }%
\left\vert f_{k}\left( u\right) \right\vert ^{\left( p-\varepsilon \right)
^{\prime }}d\lambda \left( u\right) \right) ^{\frac{1}{\left( p-\varepsilon
\right) ^{\prime }}}\right\} =\left\Vert f\right\Vert _{(p^{\prime },\theta
}.
\end{equation*}%
Thus $D_{A}^{p^{\prime }}f\in L^{(p^{\prime },\theta }\left( G\right) $.
\end{proof}

\begin{theorem}
Let $A$ be an automorphism of $G$ and $m\in BM_{\theta }\left[
(p_{1}^{\prime };(p_{2}^{\prime };(p_{3}^{\prime }\right] $ . If \ $\frac{1}{%
q}=\frac{1}{p_{1}^{\prime }}+\frac{1}{p_{2}^{\prime }}-\frac{1}{%
p_{3}^{\prime }}$,\ then $\tilde{D}_{A^{\ast }}^{q}m\in BM_{\theta }\left[
(p_{1}^{\prime };(p_{2}^{\prime };(p_{3}^{\prime }\right] ,$ where $\tilde{D}%
_{A^{\ast }}^{q}m\left( s,t\right) =\left\vert A^{\ast }\right\vert ^{\frac{1%
}{q}}$ $m\left( A^{\ast }s,A^{\ast }t\right) $. Furthermore%
\begin{equation*}
\left\Vert \tilde{D}_{A^{\ast }}^{q}m\right\Vert _{\left[ (p_{1}^{\prime
};(p_{2}^{\prime };(p_{3}^{\prime }\right] _{\theta }}\leq \left\Vert
m\right\Vert _{\left[ (p_{1}^{\prime };(p_{2}^{\prime };(p_{3}^{\prime }%
\right] _{\theta }}.
\end{equation*}
\end{theorem}

\begin{proof}
Take any$\ f\in L^{(p_{1}^{\prime },\theta }$ $\left( G\right) $ and $g\in
L^{(p_{2}^{\prime },\theta }\left( G\right) $. We know by Lemma 2 that $%
D_{A}^{p_{1}}f\in L^{(p_{1}^{\prime },\theta }$ $\left( G\right) $ and $%
D_{A}^{p_{2}}g\in L^{(p_{2}^{\prime },\theta }\left( G\right) $. If we put $%
A^{\ast }s=u$ and $A^{\ast }t=v$, then $d\mu \left( u\right) =\left\vert
A^{\ast }\right\vert d\mu \left( s\right) $ and $d\mu \left( v\right)
=\left\vert A^{\ast }\right\vert d\mu \left( t\right) .$ From the assumption 
$\frac{1}{q}=\frac{1}{p_{1}^{\prime }}+\frac{1}{p_{2}^{\prime }}-\frac{1}{%
p_{3}^{\prime }},$ we have 
\begin{equation*}
B_{\tilde{D}_{A^{\ast }}^{q}m}\left( f,g\right) \left( x\right)
=\dsum\limits_{s\in \hat{G}}\dsum\limits_{t\in \hat{G}}\hat{f}\left(
s\right) \hat{g}\left( t\right) \tilde{D}_{A^{\ast }}^{q}m\left( s,t\right)
\left\langle s+t,x\right\rangle
\end{equation*}%
\begin{equation*}
=\dsum\limits_{s\in \hat{G}}\dsum\limits_{t\in \hat{G}}\hat{f}\left(
s\right) \hat{g}\left( t\right) \left\vert A^{\ast }\right\vert ^{\frac{1}{q}%
}m\left( A^{\ast }s,A^{\ast }t\right) \left\langle s+t,x\right\rangle
\end{equation*}%
\begin{equation*}
=\left\vert A^{\ast }\right\vert ^{-1}\dsum\limits_{u\in \hat{G}}\left\vert
A^{\ast }\right\vert ^{-1}\dsum\limits_{v\in \hat{G}}\hat{f}\left( A^{\ast
-1}u\right) \hat{g}\left( A^{\ast -1}v\right) \left\vert A^{\ast
}\right\vert ^{\frac{1}{q}}m\left( u,v\right) \left\langle A^{\ast
-1}u+A^{\ast -1}v,x\right\rangle
\end{equation*}%
\begin{equation*}
=\left\vert A^{\ast }\right\vert ^{-2}\dsum\limits_{u\in \hat{G}%
}\dsum\limits_{v\in \hat{G}}\hat{f}\left( A^{\ast -1}u\right) \hat{g}\left(
A^{\ast -1}v\right) \left\vert A^{\ast }\right\vert ^{\frac{1}{q}}m\left(
u,v\right) \left\langle \left( A^{-1}\right) ^{\ast }(u+v),x\right\rangle
\end{equation*}%
\begin{equation*}
=\left\vert A^{\ast }\right\vert ^{-2}\dsum\limits_{u\in \hat{G}%
}\dsum\limits_{v\in \hat{G}}\hat{f}\left( A^{\ast -1}u\right) \hat{g}\left(
A^{\ast -1}v\right) \left\vert A^{\ast }\right\vert ^{\frac{1}{q}}m\left(
u,v\right) \left\langle u+v,A^{-1}x\right\rangle
\end{equation*}%
\begin{equation*}
=\left\vert A^{\ast }\right\vert ^{-2}\dsum\limits_{u\in \hat{G}%
}\dsum\limits_{v\in \hat{G}}\hat{f}\left( A^{\ast -1}u\right) \hat{g}\left(
A^{\ast -1}v\right) \left\vert A^{\ast }\right\vert ^{\frac{1}{p_{1}^{\prime
}}+\frac{1}{p_{2}^{\prime }}-\frac{1}{p_{3}^{\prime }}}m\left( u,v\right)
\left\langle u+v,A^{-1}x\right\rangle
\end{equation*}%
\begin{equation*}
=\left\vert A^{\ast }\right\vert ^{-2}\dsum\limits_{u\in \hat{G}%
}\dsum\limits_{v\in \hat{G}}\hat{f}\left( A^{\ast -1}u\right) \hat{g}\left(
A^{\ast -1}v\right) \left\vert A^{\ast }\right\vert ^{1-\frac{1}{p_{1}}+1-%
\frac{1}{p_{2}}-\frac{1}{p_{3}^{\prime }}}m\left( u,v\right) \left\langle
u+v,A^{-1}x\right\rangle
\end{equation*}%
\begin{equation*}
=\left\vert A^{\ast }\right\vert ^{-\frac{1}{p_{3}^{\prime }}}\left\vert
A^{\ast }\right\vert ^{-2}\left\vert A^{\ast }\right\vert
^{2}\dsum\limits_{u\in \hat{G}}\dsum\limits_{v\in \hat{G}}\left\vert A^{\ast
}\right\vert ^{-\frac{1}{p_{1}}}\hat{f}\left( A^{\ast -1}u\right) \left\vert
A^{\ast }\right\vert ^{-\frac{1}{p_{2}}}\hat{g}\left( A^{\ast -1}v\right)
m\left( u,v\right) \left\langle u+v,A^{-1}x\right\rangle
\end{equation*}%
\begin{equation}
=\left\vert A^{\ast }\right\vert ^{-\frac{1}{p_{3}^{\prime }}%
}\dsum\limits_{u\in \hat{G}}\dsum\limits_{v\in \hat{G}}D_{A^{\ast
-1}}^{p_{1}}\hat{f}\left( u\right) D_{A^{\ast -1}}^{p_{2}}\hat{g}\left(
v\right) m\left( u,v\right) \left\langle u+v,A^{-1}x\right\rangle \text{.} 
\tag{2.16}
\end{equation}%
On the other hand, if we say that $Ay=s,$ then we have 
\begin{equation*}
\left( D_{A}^{p_{1}^{\prime }}f\right) \symbol{94}\left( u\right) =\underset{%
G}{\int }D_{A}^{p_{1}^{\prime }}f\left( y\right) \left\langle
u,-y\right\rangle d\lambda \left( y\right)
\end{equation*}%
\begin{equation*}
=\underset{G}{\int }\left\vert A\right\vert ^{\frac{1}{p_{1}^{\prime }}}\hat{%
f}\left( Ay\right) \left\langle u,-y\right\rangle d\lambda \left( y\right) =%
\underset{G}{\int }\left\vert A\right\vert ^{\frac{1}{p_{1}^{\prime }}%
}f\left( s\right) \left\langle u,-A^{-1}s\right\rangle d\lambda \left(
A^{-1}s\right)
\end{equation*}%
\begin{equation*}
=\underset{G}{\int }\left\vert A\right\vert ^{\frac{1}{p_{1}^{\prime }}%
}f\left( s\right) \left\langle u,-A^{-1}s\right\rangle \left\vert
A^{-1}\right\vert d\lambda \left( y\right) =\underset{G}{\int }\left\vert
A\right\vert ^{\frac{1}{p_{1}^{\prime }}-1}f\left( s\right) \left\langle
u,-A^{-1}s\right\rangle d\lambda \left( s\right)
\end{equation*}%
\begin{equation*}
=\underset{G}{\int }\left\vert A\right\vert ^{-\frac{1}{p_{1}}}f\left(
s\right) \left\langle \left( A^{-1}\right) ^{\ast }u,-s\right\rangle
d\lambda \left( y\right) =\left\vert A^{\ast }\right\vert ^{-\frac{1}{p_{1}}}%
\underset{G}{\int }f\left( s\right) \left\langle A^{\ast
-1}u,-s\right\rangle d\lambda \left( s\right)
\end{equation*}%
\begin{equation*}
=\left\vert A^{\ast }\right\vert ^{-\frac{1}{p_{1}}}\hat{f}\left( A^{\ast
-1}u\right) =D_{A^{\ast -1}}^{p_{1}}\hat{f}\left( u\right) .
\end{equation*}%
Similarly we achieve $\left( D_{A}^{p_{2}^{\prime }}g\right) \symbol{94}%
=D_{A^{\ast -1}}^{p_{2}}\hat{g}$. Then from (2.16), we obtain%
\begin{equation*}
B_{\tilde{D}_{A^{\ast }}^{q}m}\left( f,g\right) \left( x\right) =\left\vert
A^{\ast }\right\vert ^{-\frac{1}{p_{3}^{\prime }}}\dsum\limits_{u\in \hat{G}%
}\dsum\limits_{v\in \hat{G}}D_{A^{\ast -1}}^{p_{1}}\hat{f}\left( u\right)
D_{A^{\ast -1}}^{p_{2}}\hat{g}\left( u\right) m\left( u,v\right)
\left\langle u+v,A^{-1}x\right\rangle
\end{equation*}%
\begin{equation*}
=\left\vert A^{\ast }\right\vert ^{-\frac{1}{p_{3}^{\prime }}%
}\dsum\limits_{u\in \hat{G}}\dsum\limits_{v\in \hat{G}}\left(
D_{A}^{p_{1}^{\prime }}f\right) \symbol{94}\left( u\right) \left(
D_{A}^{p_{2}^{\prime }}g\right) \symbol{94}\left( v\right) m\left(
u,v\right) \left\langle u+v,A^{-1}x\right\rangle
\end{equation*}%
\begin{equation*}
=\left\vert A\right\vert ^{-\frac{1}{p_{3}^{\prime }}}B_{m}\left(
D_{A}^{p_{1}^{\prime }}f,D_{A}^{p_{2}^{\prime }}g\right) \left(
A^{-1}x\right) =D_{A^{-1}}^{p_{3}^{\prime }}B_{m}\left( D_{A}^{p_{1}^{\prime
}}f,D_{A}^{p_{2}^{\prime }}g\right) \left( x\right) .
\end{equation*}

Since $m\in BM_{\theta }\left[ (p_{1}^{\prime };(p_{2}^{\prime
};(p_{3}^{\prime }\right] $, from Lemma 2, we have 
\begin{equation*}
\left\Vert B_{\tilde{D}_{A^{\ast }}^{q}m}\left( f,g\right) \right\Vert
_{(p_{3}^{\prime },\theta }=\left\Vert D_{A^{-1}}^{p_{3}^{\prime
}}B_{m}\left( D_{A}^{p_{1}^{\prime }}f,D_{A}^{p_{2}^{\prime }}g\right)
\right\Vert _{(p_{3}^{\prime },\theta }\leq \left\Vert B_{m}\left(
D_{A}^{p_{1}^{\prime }}f,D_{A}^{p_{2}^{\prime }}g\right) \right\Vert
_{(p_{3}^{\prime },\theta }
\end{equation*}%
\begin{equation*}
\leq \left\Vert B_{m}\right\Vert \left\Vert D_{A}^{p_{1}^{\prime
}}f\right\Vert _{(p_{1}^{\prime },\theta }\left\Vert D_{A}^{p_{2}^{\prime
}}g\right\Vert _{(p_{2}^{\prime },\theta }
\end{equation*}%
\begin{equation*}
\leq \left\Vert B_{m}\right\Vert \left\Vert f\right\Vert _{(p_{1}^{\prime
},\theta }\left\Vert g\right\Vert _{(p_{2}^{\prime },\theta }
\end{equation*}%
\begin{equation}
=\left\Vert m\right\Vert _{\left[ (p_{1}^{\prime };(p_{2}^{\prime
};(p_{3}^{\prime }\right] _{\theta }}\left\Vert f\right\Vert
_{(p_{1}^{\prime },\theta }\left\Vert g\right\Vert _{(p_{2}^{\prime },\theta
}.  \tag{2.17}
\end{equation}%
Thus we obtain $\tilde{D}_{A^{\ast }}^{q}m\in BM_{\theta }\left[
(p_{1}^{\prime };(p_{2}^{\prime };(p_{3}^{\prime }\right] $. Also by (2.17),%
\begin{equation*}
\left\Vert \tilde{D}_{A^{\ast }}^{q}m\right\Vert _{\left[ (p_{1}^{\prime
};(p_{2}^{\prime };(p_{3}^{\prime }\right] _{\theta }}\leq \left\Vert
m\right\Vert _{\left[ (p_{1}^{\prime };(p_{2}^{\prime };(p_{3}^{\prime }%
\right] _{\theta }}.
\end{equation*}
\end{proof}

\begin{theorem}
Let $A$ be an automorphism of $G$ and $m\in BM_{\theta }\left[
(p_{1}^{\prime };(p_{2}^{\prime };(p_{3}^{\prime }\right] $ such that $%
m\left( A^{\ast }s,A^{\ast }t\right) =m\left( s,t\right) ,$ where $\frac{1}{q%
}=\frac{1}{p_{1}^{\prime }}+\frac{1}{p_{2}^{\prime }}-\frac{1}{p_{3}^{\prime
}}$. Then 
\begin{equation*}
\frac{1}{p_{1}^{\prime }}+\frac{1}{p_{2}^{\prime }}=\frac{1}{p_{3}^{\prime }}%
.
\end{equation*}
\end{theorem}

\begin{proof}
Assume that $f\in L^{(p_{1}^{\prime },\theta }$ $\left( G\right) $ and $g\in
L^{(p_{2}^{\prime },\theta }\left( G\right) $. Since $\frac{1}{q}=\frac{1}{%
p_{1}^{\prime }}+\frac{1}{p_{2}^{\prime }}-\frac{1}{p_{3}^{\prime }}$, then
by Theorem 3 
\begin{equation}
B_{\tilde{D}_{A^{\ast }}^{q}m}\left( f,g\right) \left( x\right)
=D_{A^{-1}}^{p_{3}^{\prime }}B_{m}\left( D_{A}^{p_{1}^{\prime
}}f,D_{A}^{p_{2}^{\prime }}g\right) \left( x\right) \text{, }x\in G\text{.} 
\tag{2.18}
\end{equation}%
On the other hand, we write 
\begin{equation*}
D_{A^{-1}}^{p_{3}^{\prime }}B_{m}\left( D_{A}^{p_{1}^{\prime
}}f,D_{A}^{p_{2}^{\prime }}g\right) \left( x\right) =
\end{equation*}%
\begin{equation*}
=\left\vert A\right\vert ^{-\frac{1}{p_{3}^{\prime }}}\dsum\limits_{u\in 
\hat{G}}\dsum\limits_{v\in \hat{G}}\left( D_{A}^{p_{1}^{\prime }}f\right) 
\symbol{94}\left( u\right) \left( D_{A}^{p_{2}^{\prime }}g\right) \symbol{94}%
\left( v\right) m\left( u,v\right) \left\langle u+v,A^{-1}x\right\rangle
\end{equation*}%
\begin{equation*}
=\left\vert A\right\vert ^{-\frac{1}{p_{3}^{\prime }}}\dsum\limits_{u\in 
\hat{G}}\dsum\limits_{v\in \hat{G}}D_{A^{\ast -1}}^{p_{1}}\hat{f}\left(
u\right) D_{A^{\ast -1}}^{p_{2}}\hat{g}\left( u\right) m\left( u,v\right)
\left\langle u+v,A^{-1}x\right\rangle
\end{equation*}%
\begin{equation*}
=\left\vert A\right\vert ^{-\frac{1}{p_{3}^{\prime }}}\dsum\limits_{u\in 
\hat{G}}\dsum\limits_{v\in \hat{G}}\left\vert A^{\ast }\right\vert ^{-\frac{1%
}{p_{1}}}\hat{f}\left( A^{\ast -1}u\right) \left\vert A^{\ast }\right\vert
^{-\frac{1}{p_{2}}}\hat{g}\left( A^{\ast -1}v\right) m\left( u,v\right)
\left\langle A^{\ast -1}\left( u+v\right) ,x\right\rangle .
\end{equation*}%
We make the substitution $A^{\ast -1}u=s$, $A^{\ast -1}v=t$. Using $\mu
\left( A^{\ast }\hat{G}\right) =\left\vert A^{\ast }\right\vert \mu \left( 
\hat{G}\right) $, $m\left( A^{\ast }s,A^{\ast }t\right) =m\left( s,t\right)
, $ $\left\vert A\right\vert =\left\vert A^{\ast }\right\vert $ and $\left(
A^{\ast }\right) ^{-1}=\left( A^{-1}\right) ^{\ast }$, we have 
\begin{equation*}
D_{A^{-1}}^{p_{3}^{\prime }}B_{m}\left( D_{A}^{p_{1}^{\prime
}}f,D_{A}^{p_{2}^{\prime }}g\right) \left( x\right) =
\end{equation*}%
\begin{equation*}
=\left\vert A\right\vert ^{-\frac{1}{p_{3}^{\prime }}-\frac{1}{p_{1}}-\frac{1%
}{p_{2}}}\left\vert A^{\ast }\right\vert \dsum\limits_{s\in \hat{G}%
}\left\vert A^{\ast }\right\vert \dsum\limits_{t\in \hat{G}}\hat{f}\left(
s\right) \hat{g}\left( t\right) m\left( A^{\ast }s,A^{\ast }t\right)
\left\langle s+t,x\right\rangle
\end{equation*}%
\begin{equation*}
=\left\vert A\right\vert ^{-\frac{1}{p_{3}^{\prime }}-\frac{1}{p_{1}}+1-%
\frac{1}{p_{2}}+1}\dsum\limits_{s\in \hat{G}}\dsum\limits_{t\in \hat{G}}\hat{%
f}\left( s\right) \hat{g}\left( t\right) m\left( s,t\right) \left\langle
s+t,x\right\rangle
\end{equation*}%
\begin{equation}
=\left\vert A\right\vert ^{\frac{1}{p_{1}^{\prime }}+\frac{1}{p_{2}^{\prime }%
}-\frac{1}{p_{3}^{\prime }}}B_{m}\left( f,g\right) \left( x\right) \text{. }
\tag{2.19}
\end{equation}%
Hence by (2.18) and (2.19), we have%
\begin{equation}
B_{m}\left( f,g\right) \left( x\right) =\left\vert A\right\vert ^{-\left( 
\frac{1}{p_{1}^{\prime }}+\frac{1}{p_{2}^{\prime }}-\frac{1}{p_{3}^{\prime }}%
\right) }B_{\tilde{D}_{A^{\ast }}^{q}m}\left( f,g\right) \left( x\right) 
\text{.}  \tag{2.20}
\end{equation}%
Since $m\in BM_{\theta }\left[ (p_{1}^{\prime };(p_{2}^{\prime
};(p_{3}^{\prime }\right] $ , by Theorem 3, we have $\tilde{D}_{A^{\ast
}}^{q}m\in BM_{\theta }\left[ (p_{1}^{\prime };(p_{2}^{\prime
};(p_{3}^{\prime }\right] $ and 
\begin{equation*}
\left\Vert \tilde{D}_{A^{\ast }}^{q}m\right\Vert _{\left[ (p_{1}^{\prime
};(p_{2}^{\prime };(p_{3}^{\prime }\right] _{\theta }}\leq \left\Vert
m\right\Vert _{\left[ (p_{1}^{\prime };(p_{2}^{\prime };(p_{3}^{\prime }%
\right] _{\theta }}.
\end{equation*}%
Then, by (2.20) and Theorem 3, 
\begin{equation*}
\left\Vert B_{m}\left( f,g\right) \right\Vert _{(p_{3}^{\prime },\theta
}=\left\vert A\right\vert ^{-\left( \frac{1}{p_{1}^{\prime }}+\frac{1}{%
p_{2}^{\prime }}-\frac{1}{p_{3}^{\prime }}\right) }\left\Vert B_{\tilde{D}%
_{A^{\ast }}^{q}m}\left( f,g\right) \right\Vert _{(p_{3}^{\prime },\theta }
\end{equation*}%
\begin{equation*}
\leq \left\vert A\right\vert ^{-\left( \frac{1}{p_{1}^{\prime }}+\frac{1}{%
p_{2}^{\prime }}-\frac{1}{p_{3}^{\prime }}\right) }\left\Vert \tilde{D}%
_{A^{\ast }}^{q}m\right\Vert _{\left[ (p_{1}^{\prime };(p_{2}^{\prime
};(p_{3}^{\prime }\right] _{\theta }}\left\Vert f\right\Vert
_{(p_{1}^{\prime },\theta }\left\Vert g\right\Vert _{(p_{2}^{\prime },\theta
}
\end{equation*}%
\begin{equation*}
\leq \left\vert A\right\vert ^{-\left( \frac{1}{p_{1}^{\prime }}+\frac{1}{%
p_{2}^{\prime }}-\frac{1}{p_{3}^{\prime }}\right) }\left\Vert m\right\Vert _{%
\left[ (p_{1}^{\prime };(p_{2}^{\prime };(p_{3}^{\prime }\right] _{\theta
}}\left\Vert f\right\Vert _{(p_{1}^{\prime },\theta }\left\Vert g\right\Vert
_{(p_{2}^{\prime },\theta }
\end{equation*}%
\begin{equation*}
=\left\vert A\right\vert ^{-\left( \frac{1}{p_{1}^{\prime }}+\frac{1}{%
p_{2}^{\prime }}-\frac{1}{p_{3}^{\prime }}\right) }\left\Vert
B_{m}\right\Vert \left\Vert f\right\Vert _{(p_{1}^{\prime },\theta
}\left\Vert g\right\Vert _{(p_{2}^{\prime },\theta }\text{.}
\end{equation*}%
Since this inequality holds for any $\left\vert A\right\vert $, one needs $%
\frac{1}{p_{1}^{\prime }}+\frac{1}{p_{2}^{\prime }}=\frac{1}{p_{3}^{\prime }}
$.
\end{proof}

\begin{theorem}
Let $m\in BM_{\theta }\left[ (p_{1}^{\prime };(p_{2}^{\prime
};(p_{3}^{\prime }\right] $.

\textbf{a) }If $\Phi \in \ell ^{1}\left( \hat{G}\times \hat{G}\right) $,
then $\Phi \ast m\in BM_{\theta }\left[ (p_{1}^{\prime };(p_{2}^{\prime
};(p_{3}^{\prime }\right] $ and 
\begin{equation*}
\left\Vert \Phi \ast m\right\Vert _{\left[ (p_{1}^{\prime };(p_{2}^{\prime
};(p_{3}^{\prime }\right] _{\theta }}\leq \left\Vert \Phi \right\Vert _{\ell
^{1}}\left\Vert m\right\Vert _{\left[ (p_{1}^{\prime };(p_{2}^{\prime
};(p_{3}^{\prime }\right] _{\theta }}\text{.}
\end{equation*}

\textbf{b)} If \ $\Phi \in L^{1}\left( G\times G\right) ,$ then $\Phi 
\symbol{94}m\in BM_{\theta }\left[ (p_{1}^{\prime };(p_{2}^{\prime
};(p_{3}^{\prime }\right] $ and 
\begin{equation*}
\left\Vert \Phi \symbol{94}m\right\Vert _{\left[ (p_{1}^{\prime
};(p_{2}^{\prime };(p_{3}^{\prime }\right] _{\theta }}\leq \left\Vert \Phi
\right\Vert _{1}\left\Vert m\right\Vert _{\left[ (p_{1}^{\prime
};(p_{2}^{\prime };(p_{3}^{\prime }\right] _{\theta }}\text{.}
\end{equation*}

\textbf{c)}\ Let $\frac{1}{p^{\prime }}=\frac{1}{p_{1}^{\prime }}+$ $\frac{1%
}{p_{2}^{\prime }}$, $p_{3}.p_{3}^{\prime }<p^{\prime }+1$ and $m\left(
s,t\right) =a$. Then $m\in BM_{\theta }\left[ (p_{1}^{\prime
};(p_{2}^{\prime };(p_{3}^{\prime }\right] .$
\end{theorem}

\begin{proof}
\textbf{a) }Take any $f,$ $g\in C^{\infty }\left( G\right) $ . Then 
\begin{equation*}
B_{\Phi \ast m}\left( f,g\right) \left( x\right) =\dsum\limits_{s\in \hat{G}%
}\dsum\limits_{t\in \hat{G}}\hat{f}\left( s\right) \hat{g}\left( t\right)
\left( \Phi \ast m\right) \left( s,t\right) \left\langle s+t,x\right\rangle
\end{equation*}%
\begin{equation*}
=\dsum\limits_{s\in \hat{G}}\dsum\limits_{t\in \hat{G}}\hat{f}\left(
s\right) \hat{g}\left( t\right) \left( \dsum\limits_{u\in \hat{G}%
}\dsum\limits_{v\in \hat{G}}m\left( s-u,t-v\right) \Phi \left( u,v\right)
\right) \left\langle s+t,x\right\rangle
\end{equation*}%
\begin{equation*}
=\dsum\limits_{u\in \hat{G}}\dsum\limits_{v\in \hat{G}}\left(
\dsum\limits_{s\in \hat{G}}\dsum\limits_{t\in \hat{G}}\hat{f}\left( s\right) 
\hat{g}\left( t\right) m\left( s-u,t-v\right) \left\langle
s+t,x\right\rangle \right) \Phi \left( u,v\right)
\end{equation*}%
\begin{equation}
=\dsum\limits_{u\in \hat{G}}\dsum\limits_{v\in \hat{G}}B_{T_{\left(
u,v\right) }m}\left( f,g\right) \left( x\right) \Phi \left( u,v\right) . 
\tag{2.21}
\end{equation}%
Since $m\in BM_{\theta }\left[ (p_{1}^{\prime };(p_{2}^{\prime
};(p_{3}^{\prime }\right] $, by Theorem 2 \ we have $T_{\left( u,v\right)
}m\in BM_{\theta }\left[ (p_{1}^{\prime };(p_{2}^{\prime };(p_{3}^{\prime }%
\right] $. Using the equality (2.21), we write \ 
\begin{equation*}
\left\Vert B_{\Phi \ast m}\left( f,g\right) \right\Vert _{(p_{3}^{\prime
},\theta }\leq \dsum\limits_{u\in \hat{G}}\dsum\limits_{v\in \hat{G}%
}\left\Vert B_{T_{\left( u,v\right) }m}\left( f,g\right) \left( x\right)
\Phi \left( u,v\right) \right\Vert _{(p_{3}^{\prime },\theta }
\end{equation*}%
\begin{equation*}
\leq \dsum\limits_{u\in \hat{G}}\dsum\limits_{v\in \hat{G}}\left\vert \Phi
\left( u,v\right) \right\vert \left\Vert T_{\left( u,v\right) }m\right\Vert
_{\left[ (p_{1}^{\prime };(p_{2}^{\prime };(p_{3}^{\prime }\right] _{\theta
}}\left\Vert f\right\Vert _{(p_{1}^{\prime },\theta }\left\Vert g\right\Vert
_{(p_{2}^{\prime },\theta }
\end{equation*}%
\begin{equation}
=\left\Vert m\right\Vert _{\left[ (p_{1}^{\prime };(p_{2}^{\prime
};(p_{3}^{\prime }\right] _{\theta }}\left\Vert \Phi \right\Vert _{\ell
^{1}}\left\Vert f\right\Vert _{(p_{1}^{\prime },\theta }\left\Vert
g\right\Vert _{(p_{2}^{\prime },\theta }<\infty \text{.}  \tag{2.22}
\end{equation}%
Hence $\Phi \ast m\in BM_{\theta }\left[ (p_{1}^{\prime };(p_{2}^{\prime
};(p_{3}^{\prime }\right] $, and by (2.22)%
\begin{equation*}
\left\Vert \Phi \ast m\right\Vert _{\left[ (p_{1}^{\prime };(p_{2}^{\prime
};(p_{3}^{\prime }\right] _{\theta }}\leq \left\Vert \Phi \right\Vert _{\ell
^{1}}\left\Vert m\right\Vert _{\left[ (p_{1}^{\prime };(p_{2}^{\prime
};(p_{3}^{\prime }\right] _{\theta }}\text{.}
\end{equation*}%
\newline

\textbf{b) }Let $\Phi \in L^{1}\left( G\times G\right) $. Take any $f,$ $%
g\in C^{\infty }\left( G\right) $. Then we have 
\begin{equation*}
B_{\hat{\Phi}m}\left( f,g\right) \left( x\right) =\dsum\limits_{s\in \hat{G}%
}\dsum\limits_{t\in \hat{G}}\hat{f}\left( s\right) \hat{g}\left( t\right)
\Phi \symbol{94}m\left( s,t\right) \left\langle s+t,x\right\rangle
\end{equation*}%
\begin{equation*}
=\dsum\limits_{s\in \hat{G}}\dsum\limits_{t\in \hat{G}}\hat{f}\left(
s\right) \hat{g}\left( t\right) \left( \underset{G}{\dint }\underset{G}{%
\dint }\Phi \left( u,v\right) \left\langle s,-u\right\rangle \left\langle
t,-v\right\rangle d\lambda \left( u\right) d\lambda \left( v\right) \right)
m\left( s,t\right) \left\langle s+t,x\right\rangle
\end{equation*}%
\begin{equation*}
=\underset{G}{\dint }\underset{G}{\dint }\Phi \left( u,v\right) \left(
\dsum\limits_{s\in \hat{G}}\dsum\limits_{t\in \hat{G}}\hat{f}\left( s\right) 
\hat{g}\left( t\right) \left\langle s,-u\right\rangle \left\langle
t,-v\right\rangle m\left( s,t\right) \left\langle s+t,x\right\rangle \right)
d\lambda \left( u\right) d\lambda \left( v\right)
\end{equation*}%
\begin{equation*}
=\underset{G}{\dint }\underset{G}{\dint }\Phi \left( u,v\right) \left(
\dsum\limits_{s\in \hat{G}}\dsum\limits_{t\in \hat{G}}\hat{f}\left( s\right) 
\hat{g}\left( t\right) M_{-v}^{2}M_{-u}^{1}m\left( s,t\right) \left\langle
s+t,x\right\rangle \right) d\lambda \left( u\right) d\lambda \left( v\right)
\end{equation*}%
\begin{equation}
=\underset{G}{\dint }\underset{G}{\dint }\Phi \left( u,v\right)
B_{M_{-v}^{2}M_{-u}^{1}}\left( f,g\right) \left( x\right) d\lambda \left(
u\right) d\lambda \left( v\right) .  \tag{2.23}
\end{equation}%
Since $m\in BM_{\theta }\left[ (p_{1}^{\prime };(p_{2}^{\prime
};(p_{3}^{\prime }\right] $, by Theorem 2 we have $M_{-v}^{2}M_{-u}^{1}m\in
BM_{\theta }\left[ (p_{1}^{\prime };(p_{2}^{\prime };(p_{3}^{\prime }\right]
,$ where $M_{-v}^{2}$ and $M_{-u}^{1}$ are modulation operators. Then by the
equality (2.23), we obtain 
\begin{equation*}
\left\Vert B_{\Phi \symbol{94}m}\left( f,g\right) \right\Vert
_{(p_{3}^{\prime },\theta }\leq \underset{G}{\dint }\underset{G}{\dint }%
\left\Vert \Phi \left( u,v\right) B_{M_{-v}^{2}M_{-u}^{1}}\left( f,g\right)
\right\Vert _{(p_{3}^{\prime },\theta }d\lambda \left( u\right) d\lambda
\left( v\right)
\end{equation*}%
\begin{equation*}
\leq \underset{G}{\dint }\underset{G}{\dint }\left\vert \Phi \left(
u,v\right) \right\vert \left\Vert M_{-v}^{2}M_{-u}^{1}m\right\Vert _{\left[
(p_{1}^{\prime };(p_{2}^{\prime };(p_{3}^{\prime }\right] _{\theta
}}\left\Vert f\right\Vert _{(p_{1}^{\prime },\theta }\left\Vert g\right\Vert
_{(p_{2}^{\prime },\theta }d\lambda \left( u\right) d\lambda \left( v\right)
\end{equation*}%
\begin{equation*}
=\underset{G}{\dint }\underset{G}{\dint }\left\vert \Phi \left( u,v\right)
\right\vert \left\Vert m\right\Vert _{\left[ (p_{1}^{\prime };(p_{2}^{\prime
};(p_{3}^{\prime }\right] _{\theta }}\left\Vert f\right\Vert
_{(p_{1}^{\prime },\theta }\left\Vert g\right\Vert _{(p_{2}^{\prime },\theta
}d\lambda \left( u\right) d\lambda \left( v\right)
\end{equation*}%
\begin{equation}
=\left\Vert m\right\Vert _{\left[ (p_{1}^{\prime };(p_{2}^{\prime
};(p_{3}^{\prime }\right] _{\theta }}\left\Vert \Phi \right\Vert
_{1}\left\Vert f\right\Vert _{(p_{1}^{\prime },\theta }\left\Vert
g\right\Vert _{(p_{2}^{\prime },\theta }.  \tag{2.24}
\end{equation}%
Thus $\Phi \symbol{94}m\in BM_{\theta }\left[ (p_{1}^{\prime
};(p_{2}^{\prime };(p_{3}^{\prime }\right] $. By (2.24), we achieve that%
\begin{equation*}
\left\Vert \Phi \symbol{94}m\right\Vert _{\left[ (p_{1}^{\prime
};(p_{2}^{\prime };(p_{3}^{\prime }\right] _{\theta }}\leq \left\Vert \Phi
\right\Vert _{1}\left\Vert m\right\Vert _{\left[ (p_{1}^{\prime
};(p_{2}^{\prime };(p_{3}^{\prime }\right] _{\theta }}\text{.}
\end{equation*}

\textbf{c)} \bigskip Take any $f\in L^{(p_{1}^{\prime },\theta }$ $\left(
G\right) $ and $g\in L^{(p_{2}^{\prime },\theta }\left( G\right) $. Then by
Lemma 1, we have%
\begin{equation*}
\left\Vert B_{m}\left( f,g\right) \right\Vert _{(p_{3}^{\prime },\theta
}=\left\Vert \dsum\limits_{s\in \hat{G}}\dsum\limits_{t\in \hat{G}}\hat{f}%
\left( s\right) \hat{g}\left( t\right) m\left( s,t\right) \left\langle
s+t,x\right\rangle \right\Vert _{(p_{3}^{\prime },\theta }
\end{equation*}%
\begin{equation*}
=\left\Vert a\dsum\limits_{s\in \hat{G}}\dsum\limits_{t\in \hat{G}}\hat{f}%
\left( s\right) \hat{g}\left( t\right) \left\langle s+t,x\right\rangle
\right\Vert _{(p_{3}^{\prime },\theta }=\left\vert a\right\vert \left\Vert
\dsum\limits_{s\in \hat{G}}\dsum\limits_{t\in \hat{G}}\hat{f}\left( s\right) 
\hat{g}\left( t\right) \left\langle s,x\right\rangle \left\langle
t,x\right\rangle \right\Vert _{(p_{3}^{\prime },\theta }
\end{equation*}%
\begin{equation}
=\left\vert a\right\vert \left\Vert \left( \dsum\limits_{s\in \hat{G}}\hat{f}%
\left( s\right) \left\langle s,x\right\rangle \right) \left(
\dsum\limits_{t\in \hat{G}}\hat{g}\left( t\right) \left\langle
t,x\right\rangle \right) \right\Vert _{(p_{3}^{\prime },\theta }=\left\vert
a\right\vert \left\Vert fg\right\Vert _{(p_{3}^{\prime },\theta }. 
\tag{2.25}
\end{equation}

Then by Lemma 1, and (2.25) we obtain 
\begin{equation*}
\left\Vert B_{m}\left( f,g\right) \right\Vert _{(p_{3}^{\prime },\theta
}\leq C\left\vert a\right\vert \left\Vert f\right\Vert _{(p_{1}^{\prime
},\theta }\left\Vert g\right\Vert _{(p_{2}^{\prime },\theta }.
\end{equation*}%
Thus $m\in BM_{\theta }\left[ (p_{1}^{\prime };(p_{2}^{\prime
};(p_{3}^{\prime }\right] .$
\end{proof}

\begin{corollary}
Let $\frac{1}{p^{\prime }}=\frac{1}{p_{1}^{\prime }}+$ $\frac{1}{%
p_{2}^{\prime }}$ , $p_{3}.p_{3}^{\prime }<p^{\prime }+1$. If \ $\Phi \in
L^{1}\left( G\times G\right) ,$ then $\Phi \symbol{94}\in BM_{\theta }\left[
(p_{1}^{\prime };(p_{2}^{\prime };(p_{3}^{\prime }\right] .$
\end{corollary}

\begin{proof}
Let $\frac{1}{p^{\prime }}=\frac{1}{p_{1}^{\prime }}+$ $\frac{1}{%
p_{2}^{\prime }}$ , $p_{3}.p_{3}^{\prime }<p^{\prime }+1$. If we take $%
m\left( s,t\right) =1$ in Theorem 5 (c), we have $m\in BM_{\theta }\left[
(p_{1}^{\prime };(p_{2}^{\prime };(p_{3}^{\prime }\right] $. Since $\Phi \in
L^{1}\left( G\times G\right) $, by Theorem 5 (b), we obtain $\Phi \symbol{94}%
=\Phi \symbol{94}m\in BM_{\theta }\left[ (p_{1}^{\prime };(p_{2}^{\prime
};(p_{3}^{\prime }\right] $.
\end{proof}

\begin{proposition}
let $A$ be an automorphism of $G$ and let $m\in BM_{\theta }\left[
(p_{1}^{\prime };(p_{2}^{\prime };(p_{3}^{\prime }\right] $. If \ $\Psi \in
\ell ^{1}\left( \hat{G},\left\vert A^{\ast }\right\vert ^{-\frac{1}{q}}d\mu
\right) $ such that $\frac{1}{q}=\frac{1}{p_{1}^{\prime }}+\frac{1}{%
p_{2}^{\prime }}-\frac{1}{p_{3}^{\prime }}$, then 
\begin{equation*}
m_{\Psi }\left( s,t\right) =\dsum\limits_{u\in A^{\ast }\hat{G}}m\left(
A^{\ast }s,A^{\ast }t\right) \Psi \left( u\right) \in BM_{\theta }\left[
(p_{1}^{\prime };(p_{2}^{\prime };(p_{3}^{\prime }\right] .
\end{equation*}%
Moreover,%
\begin{equation*}
\left\Vert m_{\Psi }\right\Vert _{\left[ (p_{1}^{\prime };(p_{2}^{\prime
};(p_{3}^{\prime }\right] _{\theta }}\leq \left\Vert \Psi \right\Vert _{\ell
^{1}\left( \hat{G},\left\vert A^{\ast }\right\vert ^{-\frac{1}{q}}d\mu
\right) }\left\Vert m\right\Vert _{\left[ (p_{1}^{\prime };(p_{2}^{\prime
};(p_{3}^{\prime }\right] _{\theta }}\text{.}
\end{equation*}
\end{proposition}

\begin{proof}
Let \ $f$, $g\in C^{\infty }\left( G\right) $ . Then 
\begin{equation*}
B_{m_{\Psi }}\left( f,g\right) \left( x\right) =\dsum\limits_{s\in \hat{G}%
}\dsum\limits_{t\in \hat{G}}\hat{f}\left( s\right) \hat{g}\left( t\right)
m_{\Psi }\left( s,t\right) \left\langle s+t,x\right\rangle
\end{equation*}%
\begin{equation*}
=\dsum\limits_{s\in \hat{G}}\dsum\limits_{t\in \hat{G}}\hat{f}\left(
s\right) \hat{g}\left( t\right) \left\{ \dsum\limits_{u\in A^{\ast }\hat{G}%
}m\left( A^{\ast }s,A^{\ast }t\right) \Psi \left( u\right) \right\}
\left\langle s+t,x\right\rangle
\end{equation*}%
\begin{equation*}
=\dsum\limits_{s\in \hat{G}}\dsum\limits_{t\in \hat{G}}\hat{f}\left(
s\right) \hat{g}\left( t\right) \left\{ \dsum\limits_{u\in A^{\ast }\hat{G}%
}\left\vert A^{\ast }\right\vert ^{-\frac{1}{q}}\left\vert A^{\ast
}\right\vert ^{\frac{1}{q}}m\left( A^{\ast }s,A^{\ast }t\right) \Psi \left(
u\right) \right\} \left\langle s+t,x\right\rangle
\end{equation*}%
\begin{equation*}
=\dsum\limits_{s\in \hat{G}}\dsum\limits_{t\in \hat{G}}\hat{f}\left(
s\right) \hat{g}\left( t\right) \left\{ \dsum\limits_{u\in A^{\ast }\hat{G}%
}\left\vert A^{\ast }\right\vert ^{-\frac{1}{q}}\tilde{D}_{A^{\ast
}}^{q}m\left( s,t\right) \Psi \left( u\right) \right\} \left\langle
s+t,x\right\rangle
\end{equation*}%
\begin{equation*}
=\dsum\limits_{u\in A^{\ast }\hat{G}}\left( \dsum\limits_{s\in \hat{G}%
}\dsum\limits_{t\in \hat{G}}\hat{f}\left( s\right) \hat{g}\left( t\right) 
\tilde{D}_{A^{\ast }}^{q}m\left( s,t\right) \left\langle s+t,x\right\rangle
\right) \left\vert A^{\ast }\right\vert ^{-\frac{1}{q}}\Psi \left( u\right)
\end{equation*}%
\begin{equation*}
=\dsum\limits_{u\in A^{\ast }\hat{G}}B_{\tilde{D}_{A^{\ast }}^{q}m}\left(
f,g\right) \Psi \left( t\right) \left\vert A^{\ast }\right\vert ^{-\frac{1}{q%
}}\Psi \left( u\right) .
\end{equation*}%
Since $m\in BM_{\theta }\left[ (p_{1}^{\prime };(p_{2}^{\prime
};(p_{3}^{\prime }\right] $, $\tilde{D}_{A^{\ast }}^{q}m\in BM\left(
\upsilon _{s_{1}},\upsilon _{s_{2}},\upsilon _{s_{3}}\right) $ by Theorem 3,
thus we observe that%
\begin{equation*}
\left\Vert B_{m_{\Psi }}\left( f,g\right) \right\Vert _{(p_{3}^{\prime
},\theta }\leq \dsum\limits_{u\in A^{\ast }\hat{G}}\left\Vert B_{\tilde{D}%
_{A^{\ast }}^{q}m}\left( f,g\right) \right\Vert _{(p_{3}^{\prime },\theta
}\left\vert A^{\ast }\right\vert ^{-\frac{1}{q}}\Psi \left( u\right)
\end{equation*}%
\begin{equation*}
\leq \dsum\limits_{u\in A^{\ast }\hat{G}}\left\Vert B_{\tilde{D}_{A^{\ast
}}^{q}m}\right\Vert \left\Vert f\right\Vert _{(p_{1}^{\prime },\theta
}\left\Vert g\right\Vert _{(p_{2}^{\prime },\theta }\left\vert A^{\ast
}\right\vert ^{-\frac{1}{q}}\Psi \left( u\right)
\end{equation*}%
\begin{equation*}
=\dsum\limits_{u\in A^{\ast }\hat{G}}\left\Vert \tilde{D}_{A^{\ast
}}^{q}m\right\Vert _{\left[ (p_{1}^{\prime };(p_{2}^{\prime };(p_{3}^{\prime
}\right] _{\theta }}\left\Vert f\right\Vert _{(p_{1}^{\prime },\theta
}\left\Vert g\right\Vert _{(p_{2}^{\prime },\theta }\left\vert A^{\ast
}\right\vert ^{-\frac{1}{q}}\Psi \left( u\right)
\end{equation*}%
\begin{equation*}
\leq \dsum\limits_{u\in A^{\ast }\hat{G}}\left\Vert m\right\Vert _{\left[
(p_{1}^{\prime };(p_{2}^{\prime };(p_{3}^{\prime }\right] _{\theta
}}\left\Vert f\right\Vert _{(p_{1}^{\prime },\theta }\left\Vert g\right\Vert
_{(p_{2}^{\prime },\theta }\left\vert A^{\ast }\right\vert ^{-\frac{1}{q}%
}\Psi \left( u\right)
\end{equation*}%
\begin{equation*}
=\left\Vert m\right\Vert _{\left[ (p_{1}^{\prime };(p_{2}^{\prime
};(p_{3}^{\prime }\right] _{\theta }}\left\Vert f\right\Vert
_{(p_{1}^{\prime },\theta }\left\Vert g\right\Vert _{(p_{2}^{\prime },\theta
}\dsum\limits_{u\in A^{\ast }\hat{G}}\left\vert A^{\ast }\right\vert ^{-%
\frac{1}{q}}\Psi \left( u\right)
\end{equation*}%
\begin{equation*}
=\left\Vert m\right\Vert _{\left[ (p_{1}^{\prime };(p_{2}^{\prime
};(p_{3}^{\prime }\right] _{\theta }}\left\Vert f\right\Vert
_{(p_{1}^{\prime },\theta }\left\Vert g\right\Vert _{(p_{2}^{\prime },\theta
}\dsum\limits_{u\in \hat{G}}\left\vert A^{\ast }\right\vert ^{-\frac{1}{q}%
}\Psi \left( u\right)
\end{equation*}%
\begin{equation*}
=\left\Vert m\right\Vert _{\left[ (p_{1}^{\prime };(p_{2}^{\prime
};(p_{3}^{\prime }\right] _{\theta }}\left\Vert f\right\Vert
_{(p_{1}^{\prime },\theta }\left\Vert g\right\Vert _{(p_{2}^{\prime },\theta
}\left\Vert \Psi \right\Vert _{\ell ^{1}\left( \hat{G},\left\vert A^{\ast
}\right\vert ^{-\frac{1}{q}}d\mu \right) }.
\end{equation*}%
Hence, $m_{\Psi }\in BM_{\theta }\left[ (p_{1}^{\prime };(p_{2}^{\prime
};(p_{3}^{\prime }\right] $ and 
\begin{equation*}
\left\Vert m_{\Psi }\right\Vert _{\left[ (p_{1}^{\prime };(p_{2}^{\prime
};(p_{3}^{\prime }\right] _{\theta }}\leq \left\Vert \Psi \right\Vert _{\ell
^{1}\left( \hat{G},\left\vert A^{\ast }\right\vert ^{-\frac{1}{q}}d\mu
\right) }\left\Vert m\right\Vert _{\left[ (p_{1}^{\prime };(p_{2}^{\prime
};(p_{3}^{\prime }\right] _{\theta }}\text{.}
\end{equation*}
\end{proof}

\begin{proposition}
Let $m\in BM_{\theta }\left[ (p_{1}^{\prime };(p_{2}^{\prime
};(p_{3}^{\prime }\right] $. If $U_{1},~U_{2}$ are bounded measurable sets
in $\hat{G}$, then 
\begin{equation*}
h\left( s,t\right) =\frac{1}{\mu \left( U_{1}\times U_{2}\right) }%
\dsum\limits_{u\in U_{1}}\dsum\limits_{v\in U_{2}}m\left( s+u,t+v\right) \in
BM_{\theta }\left[ (p_{1}^{\prime };(p_{2}^{\prime };(p_{3}^{\prime }\right] 
\text{.}
\end{equation*}
\end{proposition}

\begin{proof}
Let $f$, $g\in C^{\infty }\left( G\right) $. We observe that 
\begin{equation*}
B_{h}\left( f,g\right) \left( x\right) =\dsum\limits_{s\in \hat{G}%
}\dsum\limits_{t\in \hat{G}}\hat{f}\left( s\right) \hat{g}\left( t\right)
h\left( s,t\right) \left\langle s+t,x\right\rangle =
\end{equation*}%
\begin{equation*}
=\frac{1}{\mu \left( U_{1}\times U_{2}\right) }\dsum\limits_{u\in
U_{1}}\dsum\limits_{v\in U_{2}}\left\{ \dsum\limits_{s\in \hat{G}%
}\dsum\limits_{t\in \hat{G}}\hat{f}\left( s\right) \hat{g}\left( t\right)
m\left( s+u,t+v\right) \left\langle s+t,x\right\rangle \right\}
\end{equation*}%
\begin{equation*}
=\frac{1}{\mu \left( U_{1}\times U_{2}\right) }\dsum\limits_{u\in
U_{1}}\dsum\limits_{v\in U_{2}}B_{T_{\left( -u,-v\right) m}}\left(
f,g\right) \left( x\right) \text{.}
\end{equation*}%
By Theorem 2, we have%
\begin{equation*}
\left\Vert B_{h}\left( f,g\right) \right\Vert _{(p_{3}^{\prime },\theta
}\leq \frac{1}{\mu \left( U_{1}\times U_{2}\right) }\dsum\limits_{u\in
U_{1}}\dsum\limits_{v\in U_{2}}\left\Vert B_{T_{\left( -u,-v\right)
m}}\left( f,g\right) \right\Vert _{(p_{3}^{\prime },\theta }
\end{equation*}%
\begin{equation*}
\leq \frac{1}{\mu \left( U_{1}\times U_{2}\right) }\dsum\limits_{u\in
U_{1}}\dsum\limits_{v\in U_{2}}\left\Vert T_{\left( -u,-v\right)
m}\right\Vert _{\left[ (p_{1}^{\prime };(p_{2}^{\prime };(p_{3}^{\prime }%
\right] }\left\Vert f\right\Vert _{(p_{1}^{\prime },\theta }\left\Vert
g\right\Vert _{(p_{2}^{\prime },\theta }
\end{equation*}%
\begin{equation*}
=\frac{1}{\mu \left( U_{1}\times U_{2}\right) }\left\Vert m\right\Vert _{%
\left[ (p_{1}^{\prime };(p_{2}^{\prime };(p_{3}^{\prime }\right] }\left\Vert
f\right\Vert _{(p_{1}^{\prime },\theta }\left\Vert g\right\Vert
_{(p_{2}^{\prime },\theta }\mu \left( U_{1}\times U_{2}\right)
\end{equation*}%
\begin{equation*}
=\left\Vert m\right\Vert _{\left[ (p_{1}^{\prime };(p_{2}^{\prime
};(p_{3}^{\prime }\right] _{\theta }}\left\Vert f\right\Vert
_{(p_{1}^{\prime },\theta }\left\Vert g\right\Vert _{(p_{2}^{\prime },\theta
}\text{.}
\end{equation*}%
Finally, we achieve $h\in BM_{\theta }\left[ (p_{1}^{\prime };(p_{2}^{\prime
};(p_{3}^{\prime }\right] $.
\end{proof}

\begin{proposition}
\bigskip Let $\frac{1}{p^{\prime }}=\frac{1}{p_{1}^{\prime }}+$ $\frac{1}{%
p_{2}^{\prime }}$ , $p_{3}.p_{3}^{\prime }<p^{\prime }+1$ and let $A,$ $B$
be automorphisms of $G$. If $\lambda \in M\left( G\right) $ and $m\left(
s,t\right) =\hat{\lambda}\left( A^{\ast }s+B^{\ast }t\right) $, then $m\in
BM_{\theta }\left[ (p_{1}^{\prime };(p_{2}^{\prime };(p_{3}^{\prime }\right] 
$. Moreover,%
\begin{equation*}
\left\Vert m\right\Vert _{\left[ (p_{1}^{\prime };(p_{2}^{\prime
};(p_{3}^{\prime }\right] _{\theta }}\leq C\left\Vert \lambda \right\Vert
\end{equation*}%
for some $C>0.$
\end{proposition}

\begin{proof}
Let \ $f$, $g$, $\in C^{\infty }\left( G\right) $. Then 
\begin{equation*}
B_{m}\left( f,g\right) \left( x\right) =\dsum\limits_{s\in \hat{G}%
}\dsum\limits_{t\in \hat{G}}\hat{f}\left( s\right) \hat{g}\left( t\right)
m\left( s,t\right) \left\langle s+t,x\right\rangle
\end{equation*}%
\begin{equation*}
=\dsum\limits_{s\in \hat{G}}\dsum\limits_{t\in \hat{G}}\hat{f}\left(
s\right) \hat{g}\left( t\right) \left\{ \underset{G}{\int }\left\langle
A^{\ast }s+B^{\ast }t,-y\right\rangle d\lambda \left( y\right) \right\}
\left\langle s+t,x\right\rangle
\end{equation*}%
\begin{equation*}
=\dsum\limits_{s\in \hat{G}}\dsum\limits_{t\in \hat{G}}\hat{f}\left(
s\right) \hat{g}\left( t\right) \left\{ \underset{G}{\int }\left\langle
A^{\ast }s,-y\right\rangle \left\langle B^{\ast }t,-y\right\rangle d\lambda
\left( y\right) \right\} \left\langle s,x\right\rangle \left\langle
t,x\right\rangle
\end{equation*}%
\begin{equation*}
=\dsum\limits_{s\in \hat{G}}\dsum\limits_{t\in \hat{G}}\hat{f}\left(
s\right) \hat{g}\left( t\right) \left\{ \underset{G}{\int }\left\langle
s,A\left( -y\right) \right\rangle \left\langle t,A\left( -y\right)
\right\rangle d\lambda \left( y\right) \right\} \left\langle
s,x\right\rangle \left\langle t,x\right\rangle
\end{equation*}%
\begin{equation*}
=\underset{G}{\int }\left\{ \dsum\limits_{s\in \hat{G}}\hat{f}\left(
s\right) \left\langle s,A\left( -y\right) \right\rangle \left\langle
s,x\right\rangle \right\} \left\{ \dsum\limits_{t\in \hat{G}}\hat{g}\left(
t\right) \left\langle t,A\left( -y\right) \right\rangle \left\langle
t,x\right\rangle \right\} d\lambda \left( y\right)
\end{equation*}%
\begin{equation*}
=\underset{G}{\int }\left\{ \dsum\limits_{s\in \hat{G}}\hat{f}\left(
s\right) \left\langle s,A\left( -y\right) +x\right\rangle \right\} \left\{
\dsum\limits_{t\in \hat{G}}\hat{g}\left( t\right) \left\langle t,A\left(
-y\right) +x\right\rangle \right\} d\lambda \left( y\right)
\end{equation*}%
\begin{equation*}
=\underset{G}{\int }f\left( x+A\left( -y\right) \right) g\left( x+A\left(
-y\right) \right) d\lambda \left( y\right)
\end{equation*}%
\begin{equation}
=\underset{G}{\int }f\left( x-Ay\right) g\left( x-Ay\right) d\lambda \left(
y\right) \text{.}  \tag{2.26}
\end{equation}%
By (2.26) and Lemma 1, we have 
\begin{equation*}
\left\Vert B_{m}\left( f,g\right) \right\Vert _{(p_{3}^{\prime },\theta
}\leq \underset{G}{\int }\left\Vert f\left( x-Ay\right) g\left( x-Ay\right)
\right\Vert _{_{(p_{3}^{\prime },\theta }}d\left\vert \lambda \right\vert
\left( y\right)
\end{equation*}%
\begin{equation*}
\leq \underset{G}{\int }C\left\Vert f\left( .-Ay\right) \right\Vert
_{_{(p_{1}^{\prime },\theta }}\left\Vert g\left( .-Ay\right) \right\Vert
_{(p_{2}^{\prime },\theta }d\left\vert \lambda \right\vert \left( y\right)
\end{equation*}%
\begin{equation*}
=C\underset{G}{\int }\left\Vert f\right\Vert _{_{(p_{1}^{\prime },\theta
}}\left\Vert g\right\Vert _{(p_{2}^{\prime },\theta }d\left\vert \lambda
\right\vert \left( y\right)
\end{equation*}%
\begin{equation}
=C\left\Vert f\right\Vert _{_{(p_{1}^{\prime },\theta }}\left\Vert
g\right\Vert _{(p_{2}^{\prime },\theta }\underset{G}{\int }d\left\vert
\lambda \right\vert \left( y\right) =C\left\Vert f\right\Vert
_{_{(p_{1}^{\prime },\theta }}\left\Vert g\right\Vert _{(p_{2}^{\prime
},\theta }\left\Vert \lambda \right\Vert \text{.}  \tag{2.27}
\end{equation}%
\qquad

Since $\lambda \in M\left( G\right) $, then by (2.27) $m\in BM_{\theta }%
\left[ (p_{1}^{\prime };(p_{2}^{\prime };(p_{3}^{\prime }\right] $. Thus we
have 
\begin{equation*}
\left\Vert m\right\Vert _{\left[ (p_{1}^{\prime };(p_{2}^{\prime
};(p_{3}^{\prime }\right] _{\theta }}=\sup \left\{ \frac{\left\Vert
B_{m}\left( f,g\right) \right\Vert _{(p_{3}^{\prime },\theta }}{\left\Vert
f\right\Vert _{_{(p_{1}^{\prime },\theta }}\left\Vert g\right\Vert
_{(p_{2}^{\prime },\theta }}:\left\Vert f\right\Vert _{_{(p_{1}^{\prime
},\theta }}\leq 1\text{, }\left\Vert g\right\Vert _{(p_{2}^{\prime },\theta
}\leq 1\right\} \leq C\left\Vert \lambda \right\Vert \text{.}
\end{equation*}
\end{proof}

It is known that the unit operator $I$ is an automorphism of $G$. It is easy
to see the conjugate $I^{\ast }$ of $I$ is an unit operator from $\hat{G}$
into itself. It is continuous, one-to-one and onto. Thus, $I^{\ast }$
becomes an automorphism of $\hat{G}$. Similarly one can easily show that $-I$
and its conjugate $-$ $I^{\ast }$ are authomorphisms of $G$ and $\hat{G}$
respectively. Since $\hat{\lambda}\left( A^{\ast }s+B^{\ast }t\right) =\hat{%
\lambda}\left( s\mp t\right) $, in Proposition 3, one can get $m(s,t)=\hat{%
\lambda}\left( s\mp t\right) $. As an application of this result we can give
the following example.

\begin{example}
\bigskip If $\lambda \in M\left( G\right) $ and $m\left( s,t\right) =\hat{%
\lambda}\left( s\mp t\right) $, then $\hat{\lambda}\in $ $\tilde{M}_{\theta
}[(p_{1}^{\prime };(p_{2}^{\prime };(p_{3}^{\prime }]$ and 
\begin{equation*}
\left\Vert \hat{\lambda}\right\Vert _{\left[ (p_{1}^{\prime };(p_{2}^{\prime
};(p_{3}^{\prime }\right] _{\theta }}\leq C\left\Vert \lambda \right\Vert
_{1}\text{, }C>0
\end{equation*}%
for $\frac{1}{p_{1}^{\prime }}+\frac{1}{p_{2}^{\prime }}=\frac{1}{p^{\prime }%
}$ and $p_{3}^{\prime }p_{3}<p^{\prime }+1$.
\end{example}

\begin{theorem}
\bigskip Let $\frac{1}{p^{\prime }}=\frac{1}{p_{1}^{\prime }}+$ $\frac{1}{%
p_{2}^{\prime }}$ , $p_{3}.p_{3}^{\prime }<p^{\prime }+1$. If $m\left(
s,t\right) =\hat{\Psi}_{1}\left( s\right) \hat{\Phi}\left( s,t\right) \hat{%
\Psi}_{2}\left( t\right) $ such that $\Phi \in L^{1}\left( G\times G\right) $
and $\Psi _{1}$, $\Psi _{2}\in L^{1}\left( G\right) $, then $m\in BM_{\theta
}\left[ (p_{1}^{\prime };(p_{2}^{\prime };(p_{3}^{\prime }\right] $.
\end{theorem}

\begin{proof}
Let $f$, $g$, $h\in C^{\infty }\left( G\right) $. For the proof \ of this
Theorem we will use Theorem 1. Then by Fubini Theorem 
\begin{equation*}
\left\vert \dsum\limits_{s\in \hat{G}}\dsum\limits_{t\in \hat{G}}\hat{f}%
\left( s\right) \hat{g}\left( t\right) \hat{h}\left( s+t\right) m\left(
s,t\right) \right\vert =
\end{equation*}%
\begin{equation*}
=\left\vert \dsum\limits_{s\in \hat{G}}\dsum\limits_{t\in \hat{G}}\hat{f}%
\left( s\right) \hat{g}\left( t\right) \left\{ \underset{G}{\int }h\left(
y\right) \left\langle s+t,-y\right\rangle d\lambda \left( y\right) \right\}
m\left( s,t\right) \right\vert
\end{equation*}%
\begin{equation*}
=\left\vert \underset{G}{\int }h\left( y\right) \left\{ \dsum\limits_{s\in 
\hat{G}}\dsum\limits_{t\in \hat{G}}\hat{f}\left( s\right) \hat{g}\left(
t\right) \hat{\Psi}_{1}\left( s\right) \hat{\Phi}\left( s,t\right) \hat{\Psi}%
_{2}\left( t\right) \left\langle s+t,-y\right\rangle \right\} d\lambda
\left( y\right) \right\vert
\end{equation*}%
\begin{equation*}
=\left\vert \underset{G}{\int }h\left( y\right) \left\{ \dsum\limits_{s\in 
\hat{G}}\dsum\limits_{t\in \hat{G}}\left( f\ast \Psi _{1}\right) \symbol{94}%
\left( s\right) \left( g\ast \Psi _{2}\right) \symbol{94}\left( t\right) 
\hat{\Phi}\left( s,t\right) \left\langle s+t,-y\right\rangle \right\}
d\lambda \left( y\right) \right\vert
\end{equation*}%
\begin{equation}
\leq \underset{G}{\int }\left\vert h\left( y\right) B_{\hat{\Phi}}\left(
f\ast \Psi _{1},g\ast \Psi _{2}\right) \left( -y\right) \right\vert d\lambda
\left( y\right) \text{.}  \tag{2.28}
\end{equation}%
On the other hand $L^{(p_{1}^{\prime },\theta }$ $\left( G\right) $ and $%
L^{(p_{2}^{\prime },\theta }\left( G\right) $ are Banach convolution module
over $L^{1}\left( G\right) $, $\left[ 6\right] .$ So we have $f\ast \Psi
_{1}\in L^{(p_{1}^{\prime },\theta }$ $\left( G\right) $ and $g\ast \Psi
_{2}\in L^{(p_{2}^{\prime },\theta }\left( G\right) $. Also by Corollary 1 , 
$\hat{\Phi}\in BM_{\theta }\left[ (p_{1}^{\prime };(p_{2}^{\prime
};(p_{3}^{\prime }\right] $. Therefore we achieve $B_{\hat{\Phi}}\left(
f\ast \Psi _{1},g\ast \Psi _{2}\right) \in L^{(p_{3}^{\prime },\theta
}\left( G\right) $. By Theorem 2.5 in $\left[ 2\right] $ and the inequality
(2.28), we obtain that 
\begin{equation*}
\left\vert \dsum\limits_{s\in \hat{G}}\dsum\limits_{t\in \hat{G}}\hat{f}%
\left( s\right) \hat{g}\left( t\right) \hat{h}\left( s+t\right) m\left(
s,t\right) \right\vert \leq \left\Vert h\right\Vert _{p_{3}),\theta
}\left\Vert B_{\hat{\Phi}}\left( f\ast \Psi _{1},g\ast \Psi _{2}\right)
\right\Vert _{(p_{3}^{\prime },\theta }
\end{equation*}%
\begin{equation*}
\leq \left\Vert h\right\Vert _{p_{3}),\theta }\left\Vert B_{\hat{\Phi}%
}\right\Vert \left\Vert f\ast \Psi _{1}\right\Vert _{(p_{1}^{\prime },\theta
}\left\Vert g\ast \Psi _{2}\right\Vert _{(p_{2}^{\prime },\theta }
\end{equation*}%
\begin{equation*}
\leq \left\Vert h\right\Vert _{p_{3}),\theta }\left\Vert B_{\hat{\Phi}%
}\right\Vert \left\Vert \Psi _{1}\right\Vert _{1}\left\Vert f\right\Vert
_{(p_{1}^{\prime },\theta }\left\Vert \Psi _{2}\right\Vert _{1}\left\Vert
g\right\Vert _{(p_{2}^{\prime },\theta }\text{.}
\end{equation*}%
Let $C=\left\Vert B_{\hat{\Phi}}\right\Vert \left\Vert \Psi _{1}\right\Vert
_{1}\left\Vert \Psi _{2}\right\Vert _{1}$. Then 
\begin{equation*}
\left\vert \dsum\limits_{s\in \hat{G}}\dsum\limits_{t\in \hat{G}}\hat{f}%
\left( s\right) \hat{g}\left( t\right) \hat{h}\left( s+t\right) m\left(
s,t\right) \right\vert \leq C\left\Vert f\right\Vert _{(p_{1}^{\prime
},\theta }\left\Vert g\right\Vert _{(p_{2}^{\prime },\theta }\left\Vert
h\right\Vert _{p_{3}),\theta }\text{.}
\end{equation*}%
Hence by Theorem 1, we achieve $m\in BM_{\theta }\left[ (p_{1}^{\prime
};(p_{2}^{\prime };(p_{3}^{\prime }\right] $.
\end{proof}

\begin{example}
\bigskip If $K\in L^{1}\left( G\right) $ then $m\left( s,t\right) =\hat{K}%
\left( s-t\right) $ defines a bilinear multiplier in $BM_{\theta
}[(p_{1}^{\prime };(p_{2}^{\prime };(p_{3}^{\prime }]$ and 
\begin{equation*}
\left\Vert m\right\Vert _{\left[ (p_{1}^{\prime };(p_{2}^{\prime
};(p_{3}^{\prime }\right] _{\theta }}\leq C\left\Vert K\right\Vert _{1},~C>0
\end{equation*}
for $\frac{1}{p^{\prime }}=\frac{1}{p_{1}^{\prime }}+$ $\frac{1}{%
p_{2}^{\prime }}$ , $p_{3}.p_{3}^{\prime }<p^{\prime }+1$. \bigskip Indeed
for $f$, $g\in C^{\infty }\left( G\right) \subset $ $L^{1}\left( G\right) $,
since 
\begin{equation*}
f\left( x-y\right) =\dsum\limits_{s\in \hat{G}}\hat{f}\left( s\right)
\left\langle s,x-y\right\rangle
\end{equation*}%
and 
\begin{equation*}
g\left( x+y\right) =\dsum\limits_{t\in \hat{G}}\hat{g}\left( t\right)
\left\langle t,x+y\right\rangle ,
\end{equation*}%
one has 
\begin{equation*}
B_{m}\left( f,g\right) \left( x\right) =\dsum\limits_{s\in \hat{G}%
}\dsum\limits_{t\in \hat{G}}\hat{f}\left( s\right) \hat{g}\left( t\right)
m\left( s,t\right) \left\langle s+t,x\right\rangle
\end{equation*}%
\begin{equation*}
=\dsum\limits_{s\in \hat{G}}\dsum\limits_{t\in \hat{G}}\hat{f}\left(
s\right) \hat{g}\left( t\right) \left( \underset{G}{\int }K\left( y\right)
\left\langle s-t,-y\right\rangle d\lambda \left( y\right) \right)
\left\langle s+t,x\right\rangle
\end{equation*}%
\begin{equation*}
=\underset{G}{\int }\dsum\limits_{s\in \hat{G}}\dsum\limits_{t\in \hat{G}}%
\hat{f}\left( s\right) \hat{g}\left( t\right) K\left( y\right) \left\langle
s-t,-y\right\rangle \left\langle s+t,x\right\rangle d\lambda \left( y\right)
\end{equation*}%
\begin{equation*}
=\underset{G}{\int }\dsum\limits_{s\in \hat{G}}\dsum\limits_{t\in \hat{G}}%
\hat{f}\left( s\right) \hat{g}\left( t\right) K\left( y\right) \left\langle
s,-y\right\rangle \left\langle -t,-y\right\rangle \left\langle
s,x\right\rangle \left\langle t,x\right\rangle d\lambda \left( y\right)
\end{equation*}%
\begin{equation*}
=\underset{G}{\int }\dsum\limits_{s\in \hat{G}}\dsum\limits_{t\in \hat{G}}%
\hat{f}\left( s\right) \hat{g}\left( t\right) K\left( y\right) \left\langle
s,-y\right\rangle \left\langle t,y\right\rangle \left\langle
s,x\right\rangle \left\langle t,x\right\rangle d\lambda \left( y\right)
\end{equation*}%
\begin{equation*}
=\underset{G}{\int }\dsum\limits_{s\in \hat{G}}\dsum\limits_{t\in \hat{G}}%
\hat{f}\left( s\right) \hat{g}\left( t\right) K\left( y\right) \left\langle
s,x-y\right\rangle \left\langle t,x+y\right\rangle d\lambda \left( y\right)
\end{equation*}%
\ 
\begin{equation*}
=\underset{G}{\int }\left( \dsum\limits_{s\in \hat{G}}\hat{f}\left( s\right)
\left\langle s,x-y\right\rangle \right) \left( \dsum\limits_{t\in \hat{G}}%
\hat{g}\left( t\right) \left\langle t,x+y\right\rangle \right) K\left(
y\right) d\lambda \left( y\right)
\end{equation*}%
\begin{equation}
=\underset{G}{\int }f\left( x-y\right) g\left( x+y\right) K\left( y\right)
d\lambda \left( y\right) \text{.}  \tag{2.29}
\end{equation}%
Then from (2.29) and by Lemma 1, 
\begin{equation*}
\left\Vert B_{m}\left( f,g\right) \right\Vert _{(p_{3}^{\prime },\theta
}\leq \underset{G}{\int }\left\Vert f\left( x-y\right) g\left( x+y\right)
\right\Vert _{(p_{3}^{\prime },\theta }\left\vert K\left( y\right)
\right\vert d\lambda \left( y\right)
\end{equation*}%
\begin{equation*}
\leq C\underset{G}{\int }\left\Vert f\left( x-y\right) \right\Vert
_{(p_{1}^{\prime },\theta }\left\Vert g\left( x+y\right) \right\Vert
_{(p_{2}^{\prime },\theta }\left\vert K\left( y\right) \right\vert d\lambda
\left( y\right) ,
\end{equation*}%
for some $C>0.$ Since $\left\Vert T_{-x}f\left( -y\right) \right\Vert
_{(p_{1}^{\prime },\theta }=\left\Vert f\left( y\right) \right\Vert
_{(p_{1}^{\prime },\theta }$ and $\left\Vert T_{-x}g\left( y\right)
\right\Vert _{(p_{2}^{\prime },\theta }=\left\Vert g\left( y\right)
\right\Vert _{(p_{2}^{\prime },\theta }$ by Theorem 2, then%
\begin{equation*}
\left\Vert B_{m}\left( f,g\right) \right\Vert _{(p_{3}^{\prime },\theta
}\leq C\underset{G}{\int }\left\Vert T_{-x}f\left( -y\right) \right\Vert
_{(p_{1}^{\prime },\theta }\left\Vert T_{-x}g\left( y\right) \right\Vert
_{(p_{2}^{\prime },\theta }\left\vert K\left( y\right) \right\vert d\lambda
\left( y\right) .
\end{equation*}%
\begin{equation}
=C\underset{G}{\int }\left\Vert \tilde{f}\right\Vert _{(p_{1}^{\prime
},\theta }\left\Vert g\right\Vert _{(p_{2}^{\prime },\theta }\left\vert
K\left( y\right) \right\vert d\lambda \left( y\right) =C\left\Vert
f\right\Vert _{(p_{1}^{\prime },\theta }\left\Vert g\right\Vert
_{(p_{2}^{\prime },\theta }\left\Vert K\right\Vert _{1}  \tag{2.30}
\end{equation}%
\begin{equation*}
=C_{1}\left\Vert f\right\Vert _{(p_{1}^{\prime },\theta }\left\Vert
g\right\Vert _{(p_{2}^{\prime },\theta },
\end{equation*}%
where $C_{1}=C\left\Vert K\right\Vert _{1}$ and $\tilde{f}\left( y\right)
=f\left( -y\right) $. Thus $m\in $\ $BM_{\theta }\left[ (p_{1}^{\prime
};(p_{2}^{\prime };(p_{3}^{\prime }\right] $. Finally by using (2.30), we
obtain 
\begin{equation*}
\left\Vert m\right\Vert _{\left[ (p_{1}^{\prime };(p_{2}^{\prime
};(p_{3}^{\prime }\right] _{\theta }}=\sup \left\{ \frac{\left\Vert
B_{m}\left( f,g\right) \right\Vert _{(p_{3}^{\prime },\theta }}{\left\Vert
f\right\Vert _{(p_{1}^{\prime },\theta }\left\Vert g\right\Vert
_{(p_{2}^{\prime },\theta }}:\left\Vert f\right\Vert _{(p_{1}^{\prime
},\theta }\leq 1\text{, }\left\Vert g\right\Vert _{(p_{2}^{\prime },\theta
}\leq 1\right\} \leq C\left\Vert K\right\Vert _{1}.
\end{equation*}
\end{example}

\begin{definition}
Let$~1<p_{i}<\infty $, $p_{i}^{\prime }=\frac{p_{i}}{p_{i}-1}$, $\left(
i=1,2,3\right) $ and $\theta >0$. We denote by $\tilde{M}_{\theta
}[(p_{1}^{\prime };(p_{2}^{\prime };(p_{3}^{\prime }]$ the space of
measurable functions $M:\hat{G}\rightarrow 
\mathbb{C}
$ such that $m\left( s,t\right) =M\left( s-t\right) \in $ \ $BM_{\theta
}[(p_{1}^{\prime };(p_{2}^{\prime };(p_{3}^{\prime }]$ that is to say%
\begin{equation*}
B_{M}\left( f,g\right) \left( x\right) =\dsum\limits_{s\in \hat{G}%
}\dsum\limits_{t\in \hat{G}}\hat{f}\left( s\right) \hat{g}\left( t\right)
M\left( s-t\right) \left\langle s+t,x\right\rangle
\end{equation*}%
extends to bounded bilinear map from $L^{(p_{1}^{\prime },\theta }\left(
G\right) \times L^{(p_{2}^{\prime },\theta }\left( G\right) $ to $%
L^{(p_{3}^{\prime },\theta }\left( G\right) $. We denote $\left\Vert
M\right\Vert _{\left[ (p_{1}^{\prime };(p_{2}^{\prime };(p_{3}^{\prime }%
\right] _{\theta }}=\left\Vert B_{M}\right\Vert $.
\end{definition}

\begin{example}
Let $G$ be a locally compact abelian metric group and let $0_{G}$ be the
unit of $G$. Take the bilinear Hardy-Littlewood maximal function on $G$;

\begin{equation*}
M\left( f,g\right) \left( x\right) =\underset{r>0}{\sup }\frac{1}{\lambda
\left( B\left( 0_{G},r\right) \right) }\underset{B\left( 0_{G},r\right) }{%
\int }\left\vert f\left( x-y\right) g\left( x+y\right) \right\vert d\lambda
\left( y\right)
\end{equation*}%
for all $f,g\in L_{loc}^{1}\left( G\right) $, where $B\left( 0_{G},r\right) $
is open ball in $G$. The Hardy-Littlewood maximal function is bounded from $%
L^{(p_{1}^{\prime },\theta }\left( G\right) \times L^{(p_{2}^{\prime
},\theta }\left( G\right) $ to $L^{(p_{3}^{\prime },\theta }\left( G\right) $
whenever $\frac{1}{p^{\prime }}=\frac{1}{p_{1}^{\prime }}+$ $\frac{1}{%
p_{2}^{\prime }}$ and $p_{3}.p_{3}^{\prime }<p^{\prime }+1$.

Take the function 
\begin{equation*}
M\left( y\right) =\frac{1}{\lambda \left( B\left( 0_{G},r\right) \right) }%
\chi _{B\left( 0_{G},r\right) }\left( y\right) .
\end{equation*}%
Since $M\in L^{1}\left( G\right) $, by Proposition 4, $M$ defines a bilinear
multiplier in $\tilde{M}_{\theta }[(p_{1}^{\prime };(p_{2}^{\prime
};(p_{3}^{\prime }]$ and%
\begin{equation}
\left\Vert B_{M}\left( f,g\right) \right\Vert _{(p_{3}^{\prime },\theta
}\leq C\left\Vert f\right\Vert _{(p_{1}^{\prime },\theta }\left\Vert
g\right\Vert _{(p_{2}^{\prime },\theta }\left\Vert K\right\Vert
_{1}=C\left\Vert f\right\Vert _{(p_{1}^{\prime },\theta }\left\Vert
g\right\Vert _{(p_{2}^{\prime },\theta }  \tag{2.31}
\end{equation}%
where 
\begin{equation*}
B_{M}\left( f,g\right) \left( x\right) =\underset{G}{\int }f\left(
x-y\right) g\left( x+y\right) M\left( y\right) d\lambda \left( y\right)
\end{equation*}%
\begin{equation*}
=\frac{1}{\lambda \left( B\left( 0_{G},r\right) \right) }\underset{B\left(
0_{G},r\right) }{\int }f\left( x-y\right) g\left( x+y\right) d\lambda \left(
y\right)
\end{equation*}%
for all $r>0$. By (2.31), we write 
\begin{equation}
M\left( f,g\right) \left( x\right) =\underset{r>0}{\sup }\frac{1}{\lambda
\left( B\left( 0_{G},r\right) \right) }\underset{B\left( 0_{G},r\right) }{%
\int }\left\vert f\left( x-y\right) g\left( x+y\right) \right\vert d\lambda
\left( y\right) =\underset{r>0}{\sup }B_{M}\left( \left\vert f\right\vert
,\left\vert g\right\vert \right) .  \tag{2.32}
\end{equation}%
So by (2.31) we have 
\begin{equation}
\left\Vert B_{M}\left( \left\vert f\right\vert ,\left\vert g\right\vert
\right) \right\Vert _{(p_{3}^{\prime },\theta }\leq C\left\Vert f\right\Vert
_{(p_{1}^{\prime },\theta }\left\Vert g\right\Vert _{(p_{2}^{\prime },\theta
}  \tag{2.33}
\end{equation}%
for all $r>0$. Then by (2.32) and (2.33), we obtain%
\begin{equation*}
\left\Vert M\left( f,g\right) \right\Vert _{(p_{3}^{\prime },\theta }\leq
C\left\Vert f\right\Vert _{(p_{1}^{\prime },\theta }\left\Vert g\right\Vert
_{(p_{2}^{\prime },\theta }.
\end{equation*}%
Therefore $M\left( f,g\right) $ is bounded from $L^{(p_{1}^{\prime },\theta
}\left( G\right) \times L^{(p_{2}^{\prime },\theta }\left( G\right) $ to $%
L^{(p_{3}^{\prime },\theta }\left( G\right) .$
\end{example}

\begin{proposition}
Let $M\in \ell ^{1}\left( \hat{G}\right) $. Then for all $f\in
L^{(p_{1}^{\prime },\theta }$ $\left( G\right) $ and $g\in L^{(p_{2}^{\prime
},\theta }\left( G\right) $ 
\begin{equation*}
B_{M}\left( f,g\right) \left( x\right) =\underset{G}{\int }f\left(
x-y\right) g\left( x+y\right) M^{\vee }\left( y\right) dy.
\end{equation*}
\end{proposition}

\begin{proof}
Let $f$, $g\in C^{\infty }\left( G\right) $. Then 
\begin{equation*}
B_{M}\left( f,g\right) \left( x\right) =\dsum\limits_{s\in \hat{G}%
}\dsum\limits_{t\in \hat{G}}\hat{f}\left( s\right) \hat{g}\left( t\right)
M\left( s-t\right) \left\langle s+t,x\right\rangle
\end{equation*}%
\begin{equation*}
=\dsum\limits_{s\in \hat{G}}\dsum\limits_{t\in \hat{G}}\hat{f}\left(
s\right) \hat{g}\left( t\right) \left( \underset{G}{\int }M^{\vee }\left(
y\right) \left\langle s-t,-y\right\rangle d\lambda \left( y\right) \right)
\left\langle s+t,x\right\rangle
\end{equation*}%
\begin{equation*}
=\dsum\limits_{s\in \hat{G}}\dsum\limits_{t\in \hat{G}}\hat{f}\left(
s\right) \hat{g}\left( t\right) \left( \underset{G}{\int }M^{\vee }\left(
y\right) \left\langle s-t,-y\right\rangle \left\langle s+t,x\right\rangle
d\lambda \left( y\right) \right)
\end{equation*}%
\begin{equation*}
=\dsum\limits_{s\in \hat{G}}\dsum\limits_{t\in \hat{G}}\hat{f}\left(
s\right) \hat{g}\left( t\right) \left( \underset{G}{\int }\check{M}\left(
y\right) \left\langle s,-y\right\rangle \left\langle t,y\right\rangle
\left\langle s,x\right\rangle \left\langle t,x\right\rangle d\lambda \left(
y\right) \right)
\end{equation*}%
\begin{equation*}
=\dsum\limits_{s\in \hat{G}}\dsum\limits_{t\in \hat{G}}\hat{f}\left(
s\right) \hat{g}\left( t\right) \left( \underset{G}{\int }M^{\vee }\left(
y\right) \left\langle s,x-y\right\rangle \left\langle t,x+y\right\rangle
d\lambda \left( y\right) \right)
\end{equation*}%
\begin{equation*}
=\underset{G}{\int }\check{M}\left( y\right) \left( \dsum\limits_{s\in \hat{G%
}}\hat{f}\left( s\right) \left\langle s,x-y\right\rangle \right) \left(
\dsum\limits_{t\in \hat{G}}\hat{g}\left( t\right) \left\langle
t,x+y\right\rangle \right) d\lambda \left( y\right)
\end{equation*}%
\begin{equation*}
=\underset{G}{\int }f\left( x-y\right) g\left( x+y\right) M^{\vee }\left(
y\right) d\lambda \left( y\right) .
\end{equation*}
\end{proof}

\begin{proposition}
Let $K\in \ell ^{1}\left( \hat{G}\right) $. Then the following equalities
are satisfied;
\end{proposition}

\textbf{a)} $B_{T_{y_{1}+y_{2}}K}\left( f,g\right)
=M_{y_{1}-y_{2}}B_{K}\left( M_{-y_{1}}f,M_{y_{2}}g\right) $, $y_{1},$ $%
y_{2}\in \hat{G},$

\textbf{b)} $B_{M_{y}K}\left( f,g\right) =B_{K}\left( T_{-y}f,T_{y}g\right) $%
, $y\in \hat{G}.$

\begin{proof}
\textbf{a) }Let $f$, $g\in C^{\infty }\left( G\right) $ and let $y_{1},$ $%
y_{2}\in \hat{G}.$ If we make the substitutions $s-y_{1}=u$ and $t+y_{2}=v$,
then we have 
\begin{equation*}
B_{T_{y_{1}+y_{2}}K}\left( f,g\right) \left( x\right) =\dsum\limits_{s\in 
\hat{G}}\dsum\limits_{t\in \hat{G}}\hat{f}\left( s\right) \hat{g}\left(
t\right) T_{y_{1}+y_{2}}K\left( s-t\right) \left\langle s+t,x\right\rangle
\end{equation*}%
\begin{equation*}
=\dsum\limits_{s\in \hat{G}}\dsum\limits_{t\in \hat{G}}\hat{f}\left(
s\right) \hat{g}\left( t\right) K\left( s-t-y_{1}-y_{2}\right) \left\langle
s+t,x\right\rangle
\end{equation*}%
\begin{equation*}
=\dsum\limits_{u\in \hat{G}}\dsum\limits_{v\in \hat{G}}\hat{f}\left(
u+y_{1}\right) \hat{g}\left( v-y_{2}\right) K\left( u-v\right) \left\langle
u+v+y_{1}-y_{2},x\right\rangle
\end{equation*}%
\begin{equation*}
=\dsum\limits_{u\in \hat{G}}\dsum\limits_{v\in \hat{G}}T_{-y_{1}}\hat{f}%
\left( u\right) T_{y_{2}}\hat{g}\left( v\right) K\left( u-v\right)
\left\langle u+v,x\right\rangle \left\langle y_{1}-y_{2},x\right\rangle
\end{equation*}%
\begin{equation*}
=\left\langle y_{1}-y_{2},x\right\rangle \dsum\limits_{u\in \hat{G}%
}\dsum\limits_{v\in \hat{G}}\left( M_{-y_{1}}f\right) \symbol{94}\left(
u\right) \left( M_{y_{2}}g\right) \symbol{94}\left( v\right) K\left(
u-v\right) \left\langle u+v,x\right\rangle
\end{equation*}%
\begin{equation*}
=\left\langle y_{1}-y_{2},x\right\rangle B_{K}\left(
M_{-y_{1}}f,M_{y_{2}}g\right) \left( x\right) =M_{y_{1}-y_{2}}B_{K}\left(
M_{-y_{1}}f,M_{y_{2}}g\right) \left( x\right)
\end{equation*}

\textbf{b) }Let\textbf{\ }$f$, $g\in C^{\infty }\left( G\right) $ and let $%
y\in \hat{G}$. Then 
\begin{equation*}
B_{M_{y}K}\left( f,g\right) \left( x\right) =\dsum\limits_{s\in \hat{G}%
}\dsum\limits_{t\in \hat{G}}\hat{f}\left( s\right) \hat{g}\left( t\right)
M_{y}K\left( s-t\right) \left\langle s+t,x\right\rangle
\end{equation*}%
\begin{equation*}
=\dsum\limits_{s\in \hat{G}}\dsum\limits_{t\in \hat{G}}\hat{f}\left(
s\right) \hat{g}\left( t\right) \left\langle s-t,y\right\rangle K\left(
s-t\right) \left\langle s+t,x\right\rangle
\end{equation*}%
\begin{equation*}
=\dsum\limits_{s\in \hat{G}}\dsum\limits_{t\in \hat{G}}\hat{f}\left(
s\right) \hat{g}\left( t\right) \left\langle s,y\right\rangle \left\langle
-t,y\right\rangle K\left( s-t\right) \left\langle s+t,x\right\rangle
\end{equation*}%
\begin{equation*}
=\dsum\limits_{s\in \hat{G}}\dsum\limits_{t\in \hat{G}}\left\langle
s,y\right\rangle \hat{f}\left( s\right) \left\langle t,-y\right\rangle \hat{g%
}\left( t\right) K\left( s-t\right) \left\langle s+t,x\right\rangle
\end{equation*}%
\begin{equation*}
=\dsum\limits_{s\in \hat{G}}\dsum\limits_{t\in \hat{G}}M_{y}\hat{f}\left(
s\right) M_{-y}\hat{g}\left( t\right) K\left( s-t\right) \left\langle
s+t,x\right\rangle
\end{equation*}%
\begin{equation*}
=\dsum\limits_{s\in \hat{G}}\dsum\limits_{t\in \hat{G}}\left( T_{-y}f\right) 
\symbol{94}\left( s\right) \left( T_{y}g\right) \symbol{94}\left( t\right)
K\left( s-t\right) \left\langle s+t,x\right\rangle =B_{K}\left(
T_{-y}f,T_{y}g\right) \left( x\right) .
\end{equation*}
\end{proof}

\begin{theorem}
Let $K\in \tilde{M}_{\theta }[(p_{1}^{\prime };(p_{2}^{\prime
};(p_{3}^{\prime }]$.

\textbf{a)} If $\Phi \in \ell ^{1}\left( \hat{G}\right) $, then $\Phi \ast
K\in \tilde{M}_{\theta }[(p_{1}^{\prime };(p_{2}^{\prime };(p_{3}^{\prime }]$
and 
\begin{equation*}
\left\Vert \Phi \ast K\right\Vert _{\left[ (p_{1}^{\prime };(p_{2}^{\prime
};(p_{3}^{\prime }\right] _{\theta }}\leq \left\Vert \Phi \right\Vert _{\ell
^{1}}\left\Vert K\right\Vert _{\left[ (p_{1}^{\prime };(p_{2}^{\prime
};(p_{3}^{\prime }\right] _{\theta }}\text{.}
\end{equation*}

\textbf{b)} If $\Phi \in L^{1}\left( G\right) $, then $\hat{\Phi}K\in \tilde{%
M}_{\theta }[(p_{1}^{\prime };(p_{2}^{\prime };(p_{3}^{\prime }]$ and 
\begin{equation*}
\left\Vert \hat{\Phi}K\right\Vert _{\left[ (p_{1}^{\prime };(p_{2}^{\prime
};(p_{3}^{\prime }\right] _{\theta }}\leq \left\Vert \Phi \right\Vert
_{1}\left\Vert K\right\Vert _{\left[ (p_{1}^{\prime };(p_{2}^{\prime
};(p_{3}^{\prime }\right] _{\theta }}\text{.}
\end{equation*}
\end{theorem}

\begin{proof}
\textbf{a)} \textbf{\ }Take any $f$, $g\in C^{\infty }\left( G\right) $ and
use Proposition 5

\begin{equation*}
B_{\Phi \ast K}\left( f,g\right) \left( x\right) =\dsum\limits_{s\in \hat{G}%
}\dsum\limits_{t\in \hat{G}}\hat{f}\left( s\right) \hat{g}\left( t\right)
\left( \Phi \ast K\right) \left( s-t\right) \left\langle s+t,x\right\rangle
\end{equation*}%
\begin{equation*}
=\dsum\limits_{s\in \hat{G}}\dsum\limits_{t\in \hat{G}}\hat{f}\left(
s\right) \hat{g}\left( t\right) \left( \dsum\limits_{u\in \hat{G}}K\left(
s-t-u\right) \Phi \left( u\right) \right) \left\langle s+t,x\right\rangle
\end{equation*}%
\begin{equation*}
=\dsum\limits_{u\in \hat{G}}\left( \dsum\limits_{s\in \hat{G}%
}\dsum\limits_{t\in \hat{G}}\hat{f}\left( s\right) \hat{g}\left( t\right)
K\left( s-t-u\right) \left\langle s+t,x\right\rangle \right) \Phi \left(
u\right)
\end{equation*}%
\begin{equation*}
=\dsum\limits_{u\in \hat{G}}\left( \dsum\limits_{s\in \hat{G}%
}\dsum\limits_{t\in \hat{G}}\hat{f}\left( s\right) \hat{g}\left( t\right)
T_{u}K\left( s-t\right) \left\langle s+t,x\right\rangle \right) \Phi \left(
u\right)
\end{equation*}%
\begin{equation*}
=\dsum\limits_{u\in \hat{G}}B_{T_{u}K}\left( f,g\right) \left( x\right) \Phi
\left( u\right) =\dsum\limits_{u\in \hat{G}}B_{T_{u+0_{\hat{G}}}K}\left(
f,g\right) \left( x\right) \Phi \left( u\right)
\end{equation*}%
By Proposition 5 and the last equality%
\begin{equation*}
B_{\Phi \ast K}\left( f,g\right) \left( x\right) =\dsum\limits_{u\in \hat{G}%
}B_{T_{u+0_{\hat{G}}}K}\left( f,g\right) \left( x\right) \Phi \left( u\right)
\end{equation*}%
\begin{equation}
=\dsum\limits_{u\in \hat{G}}M_{u-0_{\hat{G}}}B_{K}\left( M_{-u}f,M_{0_{\hat{G%
}}}g\right) \left( x\right) \Phi \left( u\right) =\dsum\limits_{u\in \hat{G}%
}M_{u-0_{\hat{G}}}B_{K}\left( M_{-u}f,M_{0_{\hat{G}}}g\right) \left(
x\right) \Phi \left( u\right) .  \tag{2.34}
\end{equation}

Since $K\in \tilde{M}_{\theta }[(p_{1}^{\prime };(p_{2}^{\prime
};(p_{3}^{\prime }]$, by (2.34), we have 
\begin{equation*}
\left\Vert B_{\Phi \ast K}\left( f,g\right) \right\Vert _{(p_{3}^{\prime
},\theta }\leq \dsum\limits_{u\in \hat{G}}\left\Vert M_{u-0_{\hat{G}%
}}B_{K}\left( M_{-u}f,M_{0_{\hat{G}}}g\right) \left( x\right) \Phi \left(
u\right) \right\Vert _{(p_{3}^{\prime },\theta }
\end{equation*}%
\begin{equation*}
=\dsum\limits_{u\in \hat{G}}\left\Vert B_{K}\left( M_{-u}f,M_{0_{\hat{G}%
}}g\right) \left( x\right) \Phi \left( u\right) \right\Vert _{(p_{3}^{\prime
},\theta }
\end{equation*}%
\begin{equation*}
\leq \dsum\limits_{u\in \hat{G}}\left\vert \Phi \left( u\right) \right\vert
\left\Vert K\right\Vert _{\left[ (p_{1}^{\prime };(p_{2}^{\prime
};(p_{3}^{\prime }\right] _{\theta }}\left\Vert M_{-u}f\right\Vert
_{(p_{1}^{\prime },\theta }\left\Vert g\right\Vert _{(p_{2}^{\prime },\theta
}
\end{equation*}%
\begin{equation}
=\left\Vert M\right\Vert _{\left[ (p_{1}^{\prime };(p_{2}^{\prime
};(p_{3}^{\prime }\right] _{\theta }}\left\Vert \Phi \right\Vert _{\ell
^{1}}\left\Vert f\right\Vert _{(p_{1}^{\prime },\theta }\left\Vert
g\right\Vert _{(p_{2}^{\prime },\theta }<\infty \text{.}  \tag{2.35}
\end{equation}

Hence $\Phi \ast $ $K\in \tilde{M}_{\theta }[(p_{1}^{\prime };(p_{2}^{\prime
};(p_{3}^{\prime }]$ . Finally by (2.35), we obtain 
\begin{equation*}
\left\Vert \Phi \ast K\right\Vert _{\left[ (p_{1}^{\prime };(p_{2}^{\prime
};(p_{3}^{\prime }\right] _{\theta }}\leq \left\Vert \Phi \right\Vert _{\ell
^{1}}\left\Vert K\right\Vert _{\left[ (p_{1}^{\prime };(p_{2}^{\prime
};(p_{3}^{\prime }\right] _{\theta }}\text{.}
\end{equation*}

\textbf{b) }Let\textbf{\ } $f$, $g\in C^{\infty }\left( G\right) .$ Then by
Propositon 5

\begin{equation*}
B_{\hat{\Phi}K}\left( f,g\right) \left( x\right) =\dsum\limits_{s\in \hat{G}%
}\dsum\limits_{t\in \hat{G}}\hat{f}\left( s\right) \hat{g}\left( t\right) 
\hat{\Phi}K\left( s-t\right) \left\langle s+t,x\right\rangle
\end{equation*}%
\begin{equation*}
=\dsum\limits_{s\in \hat{G}}\dsum\limits_{t\in \hat{G}}\hat{f}\left(
s\right) \hat{g}\left( t\right) \left( \underset{G}{\dint }\Phi \left(
u\right) \left\langle s-t,-u\right\rangle d\lambda \left( u\right) \right)
K\left( s-t\right) \left\langle s+t,x\right\rangle
\end{equation*}%
\begin{equation*}
=\underset{G}{\dint }\Phi \left( u\right) \left( \dsum\limits_{s\in \hat{G}%
}\dsum\limits_{t\in \hat{G}}\hat{f}\left( s\right) \hat{g}\left( t\right)
\left\langle s-t,-u\right\rangle K\left( s-t\right) \left\langle
s+t,x\right\rangle \right) d\lambda \left( u\right)
\end{equation*}%
\begin{equation*}
=\underset{G}{\dint }\Phi \left( u\right) \left( \dsum\limits_{s\in \hat{G}%
}\dsum\limits_{t\in \hat{G}}\hat{f}\left( s\right) \hat{g}\left( t\right)
M_{-u}K\left( s-t\right) \left\langle s+t,x\right\rangle \right) d\lambda
\left( u\right)
\end{equation*}%
\begin{equation}
=\underset{G}{\dint }\Phi \left( u\right) B_{M_{-u}K}\left( f,g\right)
\left( x\right) d\lambda \left( u\right) =\underset{G}{\dint }\Phi \left(
u\right) B_{K}\left( T_{u}f,T_{-u}g\right) \left( x\right) d\lambda \left(
u\right) .  \tag{2.36}
\end{equation}

Since $K\in \tilde{M}_{\theta }[(p_{1}^{\prime };(p_{2}^{\prime
};(p_{3}^{\prime }]$, by (2.36), we obtain 
\begin{equation*}
\left\Vert B_{\hat{\Phi}K}\left( f,g\right) \right\Vert _{(p_{3}^{\prime
},\theta }\leq \underset{G}{\dint }\left\Vert \Phi \left( u\right)
B_{K}\left( T_{u}f,T_{-u}g\right) \right\Vert _{(p_{3}^{\prime },\theta
}d\lambda \left( u\right)
\end{equation*}%
\begin{equation*}
\leq \underset{G}{\dint }\left\vert \Phi \left( u\right) \right\vert
\left\Vert K\right\Vert _{\left[ (p_{1}^{\prime };(p_{2}^{\prime
};(p_{3}^{\prime }\right] _{\theta }}\left\Vert T_{u}f\right\Vert
_{(p_{1}^{\prime },\theta }\left\Vert T_{-u}g\right\Vert _{(p_{2}^{\prime
},\theta }d\lambda \left( u\right)
\end{equation*}%
\begin{equation*}
=\underset{G}{\dint }\left\vert \Phi \left( u\right) \right\vert \left\Vert
K\right\Vert _{\left[ (p_{1}^{\prime };(p_{2}^{\prime };(p_{3}^{\prime }%
\right] _{\theta }}\left\Vert f\right\Vert _{(p_{1}^{\prime },\theta
}\left\Vert g\right\Vert _{(p_{2}^{\prime },\theta }d\lambda \left( u\right)
\end{equation*}%
\begin{equation}
=\left\Vert K\right\Vert _{\left[ (p_{1}^{\prime };(p_{2}^{\prime
};(p_{3}^{\prime }\right] _{\theta }}\left\Vert \Phi \right\Vert
_{1}\left\Vert f\right\Vert _{(p_{1}^{\prime },\theta }\left\Vert
g\right\Vert _{(p_{2}^{\prime },\theta }<\infty .  \tag{2.37}
\end{equation}

Finally $\hat{\Phi}K\in \tilde{M}_{\theta }[(p_{1}^{\prime };(p_{2}^{\prime
};(p_{3}^{\prime }]$ and by (2.37) 
\begin{equation*}
\left\Vert \hat{\Phi}K\right\Vert _{\left[ (p_{1}^{\prime };(p_{2}^{\prime
};(p_{3}^{\prime }\right] _{\theta }}\leq \left\Vert \Phi \right\Vert
_{1}\left\Vert K\right\Vert _{\left[ (p_{1}^{\prime };(p_{2}^{\prime
};(p_{3}^{\prime }\right] _{\theta }}\text{.}
\end{equation*}
\end{proof}

\begin{proposition}
Let $\Phi \in L^{1}\left( G\right) $ and $M\in \tilde{M}_{\theta
}[(p_{1}^{\prime };(p_{2}^{\prime };(p_{3}^{\prime }]$. Then $m\left(
s,t\right) =M\left( s-t\right) \hat{\Phi}\left( s+t\right) \in BM_{\theta }%
\left[ (p_{1}^{\prime };(p_{2}^{\prime };(p_{3}^{\prime }\right] $ and 
\begin{equation*}
\left\Vert m\right\Vert _{\left[ (p_{1}^{\prime };(p_{2}^{\prime
};(p_{3}^{\prime }\right] _{\theta }}\leq \left\Vert \Phi \right\Vert
_{1}\left\Vert M\right\Vert _{\left[ (p_{1}^{\prime };(p_{2}^{\prime
};(p_{3}^{\prime }\right] _{\theta }}.
\end{equation*}
\end{proposition}

\begin{proof}
Let\textbf{\ } $f$, $g\in C^{\infty }\left( G\right) $, Then for all $x\in G$%
, we have 
\begin{equation*}
B_{m}\left( f,g\right) \left( x\right) =\dsum\limits_{s\in \hat{G}%
}\dsum\limits_{t\in \hat{G}}\hat{f}\left( s\right) \hat{g}\left( t\right)
M\left( s-t\right) \hat{\Phi}\left( s+t\right) \left\langle
s+t,x\right\rangle
\end{equation*}%
\begin{equation*}
=\dsum\limits_{s\in \hat{G}}\dsum\limits_{t\in \hat{G}}\hat{f}\left(
s\right) \hat{g}\left( t\right) M\left( s-t\right) \left( \underset{G}{\dint 
}\Phi \left( u\right) \left\langle s+t,-u\right\rangle d\lambda \left(
u\right) \right) \left\langle s+t,x\right\rangle
\end{equation*}%
\begin{equation*}
=\underset{G}{\dint }\Phi \left( u\right) \left( \dsum\limits_{s\in \hat{G}%
}\dsum\limits_{t\in \hat{G}}\hat{f}\left( s\right) \hat{g}\left( t\right)
M\left( s-t\right) \left\langle s+t,x-u\right\rangle \right) d\lambda \left(
u\right)
\end{equation*}%
\begin{equation}
=\underset{G}{\dint }\Phi \left( u\right) B_{M}\left( f,g\right) \left(
x-u\right) d\lambda \left( u\right) =\Phi \ast B_{M}\left( f,g\right) \left(
x\right) .  \tag{2.38}
\end{equation}%
On the other hand Since $L^{(p_{3}^{\prime },\theta }\left( G\right) $ is a
Banach convolution module over $L^{1}\left( G\right) ,$ $\left[ 6\right] $
and $M\in \tilde{M}_{\theta }[(p_{1}^{\prime };(p_{2}^{\prime
};(p_{3}^{\prime }]$, by (2.38) we have 
\begin{equation*}
\left\Vert B_{m}\left( f,g\right) \right\Vert _{(p_{3}^{\prime },\theta
}=\left\Vert \Phi \ast B_{M}\left( f,g\right) \right\Vert _{(p_{3}^{\prime
},\theta }\leq \left\Vert B_{M}\left( f,g\right) \right\Vert
_{(p_{3}^{\prime },\theta }\left\Vert \Phi \right\Vert _{1}
\end{equation*}%
\begin{equation}
\leq \left\Vert \Phi \right\Vert _{1}\left\Vert M\right\Vert _{\left[
(p_{1}^{\prime };(p_{2}^{\prime };(p_{3}^{\prime }\right] _{\theta
}}\left\Vert f\right\Vert _{(p_{1}^{\prime },\theta }\left\Vert g\right\Vert
_{(p_{2}^{\prime },\theta }<\infty .  \tag{2.39}
\end{equation}

Hence $m\in BM_{\theta }\left[ (p_{1}^{\prime };(p_{2}^{\prime
};(p_{3}^{\prime }\right] $ and by (2.39)

\begin{equation*}
\left\Vert m\right\Vert _{\left[ (p_{1}^{\prime };(p_{2}^{\prime
};(p_{3}^{\prime }\right] _{\theta }}\leq \left\Vert \Phi \right\Vert
_{1}\left\Vert M\right\Vert _{\left[ (p_{1}^{\prime };(p_{2}^{\prime
};(p_{3}^{\prime }\right] _{\theta }}\text{.}
\end{equation*}
\end{proof}

\begin{proposition}
Let $K\in \ell ^{1}\left( \hat{G}\right) $ be non-zero function and let $%
K\in \tilde{M}_{\theta }[(p_{1}^{\prime };(p_{2}^{\prime };(p_{3}^{\prime }]$%
. If $A$ is an automorphism of $G$ and if $\frac{1}{q}=\frac{1}{%
p_{1}^{\prime }}+\frac{1}{p_{2}^{\prime }}-\frac{1}{p_{3}^{\prime }}$, then
there exists $C>0$ such that%
\begin{equation*}
\left\vert \underset{G}{\int }K^{\vee }\left( u\right) d\lambda \left(
u\right) \right\vert \leq C\left\vert A\right\vert ^{\frac{1}{p_{3}^{\prime }%
}}\left\Vert K\right\Vert _{\left[ (p_{1}^{\prime };(p_{2}^{\prime
};(p_{3}^{\prime }\right] _{\theta }},\text{ }\left\vert A\right\vert <1
\end{equation*}%
and%
\begin{equation*}
\left\vert \underset{G}{\int }K^{\vee }\left( u\right) d\lambda \left(
u\right) \right\vert \leq C\left\vert A\right\vert ^{-\frac{1}{q}}\left\Vert
K\right\Vert _{\left[ (p_{1}^{\prime };(p_{2}^{\prime };(p_{3}^{\prime }%
\right] _{\theta }},\text{ }\left\vert A\right\vert \geq 1.
\end{equation*}
\end{proposition}

\begin{proof}
Let $A$ be any automorphism of $G$. Define a function $f:G\rightarrow 
\mathbb{C}
$ by $f\left( x\right) =\left\langle \gamma ,Ax\right\rangle $ for fixed $%
\gamma \in \hat{G}.$ By Proposition 4, we write%
\begin{equation*}
B_{K}\left( f,f\right) \left( x\right) =\underset{G}{\int }f\left(
x-y\right) f\left( x+y\right) K^{\vee }\left( y\right) dy
\end{equation*}%
\begin{equation*}
=\underset{G}{\int }\left\langle \gamma ,A\left( x-y\right) \right\rangle
\left\langle \gamma ,A\left( x+y\right) \right\rangle K^{\vee }\left(
y\right) dy
\end{equation*}%
\begin{equation*}
=\underset{G}{\int }\left\langle \gamma ,Ax\right\rangle \left\langle \gamma
,-Ay\right\rangle \left\langle \gamma ,Ax\right\rangle \left\langle \gamma
,Ay\right\rangle K^{\vee }\left( y\right) dy
\end{equation*}%
\begin{equation}
=\underset{G}{\int }\left\langle \gamma ,Ax\right\rangle \left\langle \gamma
,Ax\right\rangle K^{\vee }\left( y\right) dy=\left\langle \gamma
,Ax\right\rangle \left\langle \gamma ,Ax\right\rangle \underset{G}{\int }%
K^{\vee }\left( y\right) dy.  \tag{2.40}
\end{equation}%
Using (2.40) and making the substitution $Ax=u$, we have 
\begin{equation*}
\left\Vert B_{K}\left( f,f\right) \right\Vert _{p_{3}^{\prime }}=\left( 
\underset{G}{\dint }\left\vert B_{K}\left( f,f\right) \left( x\right)
\right\vert ^{p_{3}^{\prime }}d\lambda \left( x\right) \right) ^{\frac{1}{%
p_{3}^{\prime }}}
\end{equation*}%
\begin{equation*}
=\left( \underset{G}{\dint }\left\vert \left\langle \gamma ,Ax\right\rangle
\left\langle \gamma ,Ax\right\rangle \underset{G}{\int }K^{\vee }\left(
y\right) dy\right\vert ^{p_{3}^{\prime }}d\lambda \left( x\right) \right) ^{%
\frac{1}{p_{3}^{\prime }}}
\end{equation*}%
\begin{equation*}
=\left\vert \underset{G}{\int }K^{\vee }\left( y\right) dy\right\vert \left( 
\underset{G}{\dint }\left\vert \left\langle \gamma ,Ax\right\rangle
\left\langle \gamma ,Ax\right\rangle \right\vert ^{p_{3}^{\prime }}d\lambda
\left( x\right) \right) ^{\frac{1}{p_{3}^{\prime }}}
\end{equation*}%
\begin{equation*}
=\left\vert \underset{G}{\int }K^{\vee }\left( y\right) dy\right\vert \left( 
\underset{G}{\dint }\left\vert \left\langle \gamma ,u\right\rangle
\left\langle \gamma ,u\right\rangle \right\vert ^{p_{3}^{\prime }}\left\vert
A\right\vert ^{-1}d\lambda \left( u\right) \right) ^{\frac{1}{p_{3}^{\prime }%
}}
\end{equation*}%
\begin{equation}
=\left\vert A\right\vert ^{-\frac{1}{p_{3}^{\prime }}}\left\vert \underset{G}%
{\int }K^{\vee }\left( y\right) dy\right\vert \left( \underset{G}{\dint }%
d\lambda \left( u\right) \right) ^{\frac{1}{p_{3}^{\prime }}}=\left\vert
A\right\vert ^{-\frac{1}{p_{3}^{\prime }}}\lambda \left( G\right) ^{\frac{1}{%
p_{3}^{\prime }}}\left\vert \underset{G}{\int }K^{\vee }\left( y\right)
dy\right\vert .  \tag{2.41}
\end{equation}%
On the other hand we can write%
\begin{equation*}
f\left( x\right) =\left\langle \gamma ,Ax\right\rangle =\left\vert
A\right\vert ^{-\frac{1}{p_{1}^{\prime }}}\left\vert A\right\vert ^{\frac{1}{%
p_{1}^{\prime }}}\left\langle \gamma ,Ax\right\rangle =\left\vert
A\right\vert ^{-\frac{1}{p_{1}^{\prime }}}D_{A}^{p_{1}^{\prime }}\gamma
\left( x\right) .
\end{equation*}%
Let $\left\vert A\right\vert \geq 1$. By Lemma 2, we obtain 
\begin{equation*}
\left\Vert f\right\Vert _{(p_{1}^{\prime },\theta }=\left\Vert \left\vert
A\right\vert ^{-\frac{1}{p_{1}^{\prime }}}D_{A}^{p_{1}^{\prime }}\gamma
\right\Vert _{(p_{1}^{\prime },\theta }=\left\vert A\right\vert ^{-\frac{1}{%
p_{1}^{\prime }}}\left\Vert D_{A}^{p_{1}^{\prime }}\gamma \right\Vert
_{(p_{1}^{\prime },\theta }
\end{equation*}%
\begin{equation*}
=\left\vert A\right\vert ^{-\frac{1}{p_{1}^{\prime }}}\left\Vert \gamma
\right\Vert _{(p_{1}^{\prime },\theta }.
\end{equation*}%
Since $L^{p_{1}^{\prime }+\varepsilon }\left( G\right) \subset
L^{(p_{1}^{\prime },\theta }$ $\left( G\right) ,$ $[2]$, there exists $%
C_{1}>0$ such that $\left\Vert \gamma \right\Vert _{(p_{1}^{\prime },\theta
}\leq C_{1}\left\Vert \gamma \right\Vert _{p_{1}^{\prime }+\varepsilon }$.
Then 
\begin{equation*}
\left\Vert f\right\Vert _{(p_{1}^{\prime },\theta }=\left\vert A\right\vert
^{-\frac{1}{p_{1}^{\prime }}}\left\Vert \gamma \right\Vert _{(p_{1}^{\prime
},\theta }\leq C_{1}\left\vert A\right\vert ^{-\frac{1}{p_{1}^{\prime }}%
}\left\Vert \gamma \right\Vert _{p_{1}^{\prime }+\varepsilon }.
\end{equation*}%
\begin{equation*}
=C_{1}\left\vert A\right\vert ^{-\frac{1}{p_{1}^{\prime }}}\left( \underset{G%
}{\dint }\left\vert \left\langle \gamma ,x\right\rangle \right\vert
^{p_{1}^{\prime }+\varepsilon }d\lambda \left( x\right) \right) ^{\frac{1}{%
p_{1}^{\prime }+\varepsilon }}
\end{equation*}%
\begin{equation}
=C_{1}\left\vert A\right\vert ^{-\frac{1}{p_{1}^{\prime }}}\left( \underset{G%
}{\dint }d\lambda \left( x\right) \right) ^{\frac{1}{p_{1}^{\prime
}+\varepsilon }}=C_{1}\left\vert A\right\vert ^{-\frac{1}{p_{1}^{\prime }}%
}\lambda \left( G\right) ^{\frac{1}{p_{1}^{\prime }+\varepsilon }}<\infty . 
\tag{2.42}
\end{equation}%
Similarly there exists $C_{2}>0$ such that 
\begin{equation}
\left\Vert f\right\Vert _{(p_{2}^{\prime },\theta }\leq C_{2}\left\vert
A\right\vert ^{-\frac{1}{p_{2}^{\prime }}}\lambda \left( G\right) ^{\frac{1}{%
p_{2}^{\prime }+\varepsilon }}<\infty .  \tag{2.43}
\end{equation}%
Using the assumption $K\in \tilde{M}_{\theta }[(p_{1}^{\prime
};(p_{2}^{\prime };(p_{3}^{\prime }]$ and the inequalities (2.42) and
(2.43), we obtain%
\begin{equation*}
\left\Vert B_{K}\left( f,f\right) \right\Vert _{(p_{3}^{\prime },\theta
}\leq \left\Vert K\right\Vert _{\left[ (p_{1}^{\prime };(p_{2}^{\prime
};(p_{3}^{\prime }\right] _{\theta }}\left\Vert f\right\Vert
_{(p_{1}^{\prime },\theta }\left\Vert f\right\Vert _{(p_{2}^{\prime },\theta
}
\end{equation*}%
\begin{equation}
\leq C_{1}C_{2}\left\vert A\right\vert ^{-\frac{1}{p_{1}^{\prime }}-\frac{1}{%
p_{2}^{\prime }}}\lambda \left( G\right) ^{\frac{1}{p_{1}^{\prime
}+\varepsilon }+\frac{1}{p_{2}^{\prime }+\varepsilon }}\left\Vert
K\right\Vert _{\left[ (p_{1}^{\prime };(p_{2}^{\prime };(p_{3}^{\prime }%
\right] _{\theta }}.  \tag{2.44}
\end{equation}%
Since $L^{(p_{3}^{\prime },\theta }$ $\left( G\right) \subset
L^{p_{3}^{\prime }}\left( G\right) $, there exists $C_{3}>0$ such that 
\begin{equation}
\left\Vert B_{K}\left( f,f\right) \right\Vert _{p_{3}^{\prime }}\leq
C_{3}\left\Vert B_{K}\left( f,f\right) \right\Vert _{(p_{3}^{\prime },\theta
}.  \tag{2.45}
\end{equation}%
Then by (2.44) and (2.45), we have%
\begin{equation}
\left\Vert B_{K}\left( f,f\right) \right\Vert _{p_{3}^{\prime }}\leq
C_{1}C_{2}C_{3}\left\vert A\right\vert ^{-\frac{1}{p_{1}^{\prime }}-\frac{1}{%
p_{2}^{\prime }}}\lambda \left( G\right) ^{\frac{1}{p_{1}^{\prime
}+\varepsilon }+\frac{1}{p_{2}^{\prime }+\varepsilon }}\left\Vert
K\right\Vert _{\left[ (p_{1}^{\prime };(p_{2}^{\prime };(p_{3}^{\prime }%
\right] _{\theta }}  \tag{2.46}
\end{equation}%
Also by (2.41) and (2.46), we achieve%
\begin{equation*}
\lambda \left( G\right) ^{\frac{1}{p_{3}^{\prime }}}\left\vert A\right\vert
^{-\frac{1}{p_{3}^{\prime }}}\left\vert \underset{G}{\int }K^{\vee }\left(
y\right) dy\right\vert =\left\Vert B_{K}\left( f,f\right) \right\Vert
_{p_{3}^{\prime }}\leq
\end{equation*}%
\begin{equation*}
\leq C_{1}C_{2}C_{3}\left\vert A\right\vert ^{-\frac{1}{p_{1}^{\prime }}-%
\frac{1}{p_{2}^{\prime }}}\lambda \left( G\right) ^{\frac{1}{p_{1}^{\prime
}+\varepsilon }+\frac{1}{p_{2}^{\prime }+\varepsilon }}\left\Vert
K\right\Vert _{\left[ (p_{1}^{\prime };(p_{2}^{\prime };(p_{3}^{\prime }%
\right] _{\theta }}.
\end{equation*}%
This implies%
\begin{equation*}
\left\vert \underset{G}{\int }K^{\vee }\left( y\right) dy\right\vert \leq
C\left\vert A\right\vert ^{\frac{1}{p_{3}^{\prime }}-\frac{1}{p_{1}^{\prime }%
}-\frac{1}{p_{2}^{\prime }}}\left\Vert K\right\Vert _{\left[ (p_{1}^{\prime
};(p_{2}^{\prime };(p_{3}^{\prime }\right] _{\theta }}=C\left\vert
A\right\vert ^{-\frac{1}{q}}\left\Vert K\right\Vert _{\left[ (p_{1}^{\prime
};(p_{2}^{\prime };(p_{3}^{\prime }\right] _{\theta }},
\end{equation*}%
where $C=C_{1}C_{2}C_{3}\lambda \left( G\right) ^{\frac{1}{p_{1}^{\prime
}+\varepsilon }+\frac{1}{p_{2}^{\prime }+\varepsilon }-\frac{1}{%
p_{3}^{\prime }}}.$

Now let $\left\vert A\right\vert <1.$ By Lemma 2, we have 
\begin{equation*}
\left\Vert f\right\Vert _{(p_{1}^{\prime },\theta }=\left\Vert \left\vert
A\right\vert ^{-\frac{1}{p_{1}^{\prime }}}D_{A}^{p_{1}^{\prime }}\gamma
\right\Vert _{(p_{1}^{\prime },\theta }=\left\vert A\right\vert ^{-\frac{1}{%
p_{1}^{\prime }}}\left\Vert D_{A}^{p_{1}^{\prime }}\gamma \right\Vert
_{(p_{1}^{\prime },\theta }
\end{equation*}%
\begin{equation*}
=\left\vert A\right\vert ^{-\frac{1}{p_{1}^{\prime }}}\left\vert
A\right\vert ^{\frac{1}{p_{1}^{\prime }}}\left\Vert \gamma \right\Vert
_{(p_{1}^{\prime },\theta }=\left\Vert \gamma \right\Vert _{(p_{1}^{\prime
},\theta }.
\end{equation*}%
Then Since $L^{p_{1}^{\prime }+\varepsilon }\left( G\right) \subset
L^{(p_{1}^{\prime },\theta }$ $\left( G\right) ,$ we achieve 
\begin{equation*}
\left\Vert f\right\Vert _{(p_{1}^{\prime },\theta }=\left\Vert \gamma
\right\Vert _{(p_{1}^{\prime },\theta }\leq C_{1}\left\Vert \gamma
\right\Vert _{p_{1}^{\prime }+\varepsilon }.
\end{equation*}%
\begin{equation}
=C_{1}\lambda \left( G\right) ^{\frac{1}{p_{1}^{\prime }+\varepsilon }%
}<\infty ,  \tag{2.47}
\end{equation}%
for some $C_{1}>0$. Similarly we have 
\begin{equation}
\left\Vert f\right\Vert _{(p_{2}^{\prime },\theta }\leq C_{2}\lambda \left(
G\right) ^{\frac{1}{p_{2}^{\prime }+\varepsilon }}<\infty ,  \tag{2.48}
\end{equation}%
for some $C_{2}>0$. Again using the assumption $K\in \tilde{M}_{\theta
}[(p_{1}^{\prime };(p_{2}^{\prime };(p_{3}^{\prime }]$ and the inequalities
(2.47) and (2.48), we obtain%
\begin{equation*}
\left\Vert B_{K}\left( f,f\right) \right\Vert _{(p_{3}^{\prime },\theta
}\leq \left\Vert K\right\Vert _{\left[ (p_{1}^{\prime };(p_{2}^{\prime
};(p_{3}^{\prime }\right] _{\theta }}\left\Vert f\right\Vert
_{(p_{1}^{\prime },\theta }\left\Vert f\right\Vert _{(p_{2}^{\prime },\theta
}
\end{equation*}%
\begin{equation}
\leq C_{1}C_{2}\lambda \left( G\right) ^{\frac{1}{p_{1}^{\prime
}+\varepsilon }+\frac{1}{p_{2}^{\prime }+\varepsilon }}\left\Vert
K\right\Vert _{\left[ (p_{1}^{\prime };(p_{2}^{\prime };(p_{3}^{\prime }%
\right] _{\theta }}.  \tag{2.49}
\end{equation}%
Thus by (2.45) and (2.49), we have%
\begin{equation}
\left\Vert B_{K}\left( f,f\right) \right\Vert _{p_{3}^{\prime }}\leq
C_{1}C_{2}C_{3}\lambda \left( G\right) ^{\frac{1}{p_{1}^{\prime
}+\varepsilon }+\frac{1}{p_{2}^{\prime }+\varepsilon }}\left\Vert
K\right\Vert _{\left[ (p_{1}^{\prime };(p_{2}^{\prime };(p_{3}^{\prime }%
\right] _{\theta }}.  \tag{2.50}
\end{equation}%
Using (2.41) and (2.50), we achieve

\begin{equation*}
\lambda \left( G\right) ^{\frac{1}{p_{3}^{\prime }}}\left\vert A\right\vert
^{-\frac{1}{p_{3}^{\prime }}}\left\vert \underset{G}{\int }K^{\vee }\left(
y\right) dy\right\vert =\left\Vert B_{K}\left( f,f\right) \right\Vert
_{p_{3}^{\prime }}\leq
\end{equation*}%
\begin{equation*}
\leq C_{1}C_{2}C_{3}\lambda \left( G\right) ^{\frac{1}{p_{1}^{\prime
}+\varepsilon }+\frac{1}{p_{2}^{\prime }+\varepsilon }}\left\Vert
K\right\Vert _{\left[ (p_{1}^{\prime };(p_{2}^{\prime };(p_{3}^{\prime }%
\right] _{\theta }}.
\end{equation*}%
Then%
\begin{equation*}
\left\vert \underset{G}{\int }K^{\vee }\left( y\right) dy\right\vert \leq
C_{1}C_{2}C_{3}\left\vert A\right\vert ^{\frac{1}{p_{3}^{\prime }}}\lambda
\left( G\right) ^{\frac{1}{p_{1}^{\prime }+\varepsilon }+\frac{1}{%
p_{2}^{\prime }+\varepsilon }-\frac{1}{p_{3}^{\prime }}}\left\Vert
K\right\Vert _{\left[ (p_{1}^{\prime };(p_{2}^{\prime };(p_{3}^{\prime }%
\right] _{\theta }}=C\left\vert A\right\vert ^{\frac{1}{p_{3}^{\prime }}%
}\left\Vert K\right\Vert _{\left[ (p_{1}^{\prime };(p_{2}^{\prime
};(p_{3}^{\prime }\right] _{\theta }},
\end{equation*}%
where $C=C_{1}C_{2}C_{3}\lambda \left( G\right) ^{\frac{1}{p_{1}^{\prime
}+\varepsilon }+\frac{1}{p_{2}^{\prime }+\varepsilon }-\frac{1}{%
p_{3}^{\prime }}}.$
\end{proof}

\begin{proposition}
Let $K\in \ell ^{1}\left( \hat{G}\right) $ be non-zero function and let $%
K\in \tilde{M}_{\theta }[(p_{1}^{\prime };(p_{2}^{\prime };(p_{3}^{\prime }]$%
. If $A$ is an automorphism of $G$ satisfying $\left\vert A\right\vert >1,$
then%
\begin{equation*}
\frac{1}{p_{3}^{\prime }}\geq \frac{1}{p_{1}^{\prime }}+\frac{1}{%
p_{2}^{\prime }}.
\end{equation*}
\end{proposition}

\begin{proof}
Assume that $\frac{1}{p_{3}^{\prime }}<\frac{1}{p_{1}^{\prime }}+\frac{1}{%
p_{2}^{\prime }}$. By Proposition 5,%
\begin{equation*}
\left\Vert B_{T_{-u}K}\left( f,g\right) \right\Vert _{(p_{3}^{\prime
},\theta }=\left\Vert B_{T_{-u+0_{\hat{G}}}K}\left( f,g\right) \right\Vert
_{(p_{3}^{\prime },\theta }
\end{equation*}%
\begin{equation*}
=\left\Vert M_{-u-0_{\hat{G}}}B_{K}\left( M_{u}f,M_{0_{\hat{G}}}g\right)
\right\Vert _{(p_{3}^{\prime },\theta }=\left\Vert B_{K}\left(
M_{u}f,g\right) \right\Vert _{(p_{3}^{\prime },\theta }
\end{equation*}%
\begin{equation*}
\leq \left\Vert K\right\Vert _{\left[ (p_{1}^{\prime };(p_{2}^{\prime
};(p_{3}^{\prime }\right] _{\theta }}\left\Vert M_{u}f\right\Vert
_{(p_{1}^{\prime },\theta }\left\Vert g\right\Vert _{(p_{2}^{\prime },\theta
}
\end{equation*}%
\begin{equation*}
=\left\Vert K\right\Vert _{\left[ (p_{1}^{\prime };(p_{2}^{\prime
};(p_{3}^{\prime }\right] _{\theta }}\left\Vert f\right\Vert
_{(p_{1}^{\prime },\theta }\left\Vert g\right\Vert _{(p_{2}^{\prime },\theta
},
\end{equation*}%
and so $T_{-u}K\in \tilde{M}_{\theta }[(p_{1}^{\prime };(p_{2}^{\prime
};(p_{3}^{\prime }]$ for all $u\in \hat{G}$. On the other hand

\begin{equation*}
\left\Vert T_{-u}K\right\Vert _{\left[ (p_{1}^{\prime };(p_{2}^{\prime
};(p_{3}^{\prime }\right] _{\theta }}=\left\Vert B_{T_{-u}K}\right\Vert
\end{equation*}%
\begin{equation*}
=\sup \left\{ \frac{\left\Vert B_{T_{-u}K}\left( f,g\right) \right\Vert
_{(p_{3}^{\prime },\theta }}{\left\Vert f\right\Vert _{(p_{1}^{\prime
},\theta }\left\Vert g\right\Vert _{(p_{2}^{\prime },\theta }}:\left\Vert
f\right\Vert _{(p_{1}^{\prime },\theta }\leq 1\text{, }\left\Vert
g\right\Vert _{(p_{2}^{\prime },\theta }\leq 1\right\}
\end{equation*}%
\begin{equation*}
=\sup \left\{ \frac{\left\Vert M_{-u}B_{K}\left( M_{u}f,M_{0_{\hat{G}%
}}g\right) \right\Vert _{(p_{3}^{\prime },\theta }}{\left\Vert
M_{u}f\right\Vert _{(p_{1}^{\prime },\theta }\left\Vert M_{0_{\hat{G}%
}}g\right\Vert _{(p_{2}^{\prime },\theta }}:\left\Vert M_{u}f\right\Vert
_{(p_{1}^{\prime },\theta }\leq 1\text{, }\left\Vert M_{0_{\hat{G}%
}}g\right\Vert _{(p_{2}^{\prime },\theta }\leq 1\right\}
\end{equation*}%
\begin{equation*}
=\left\Vert B_{K}\right\Vert =\left\Vert K\right\Vert _{\left[
(p_{1}^{\prime };(p_{2}^{\prime };(p_{3}^{\prime }\right] _{\theta }}\text{.}
\end{equation*}%
Thus, by Proposition 7, there exists $C>0$ such that 
\begin{equation*}
\left\vert \underset{G}{\int }\left( T_{-u}K\right) ^{\vee }\left( y\right)
d\lambda \left( y\right) \right\vert \leq C\left\vert A\right\vert ^{-\frac{1%
}{q}}\left\Vert T_{-u}K\right\Vert _{\left[ (p_{1}^{\prime };(p_{2}^{\prime
};(p_{3}^{\prime }\right] _{\theta }}
\end{equation*}%
\begin{equation}
=C\left\vert A\right\vert ^{-\frac{1}{q}}\left\Vert K\right\Vert _{\left[
(p_{1}^{\prime };(p_{2}^{\prime };(p_{3}^{\prime }\right] _{\theta }}. 
\tag{2.51}
\end{equation}%
Since $T_{-u}K\in \ell ^{1}\left( \hat{G}\right) $, from (2.51) we write%
\begin{equation*}
\left\vert K\left( u\right) \right\vert =\left\vert T_{-u}K\left( 0_{\hat{G}%
}\right) \right\vert =\left\vert \left( \left( T_{-u}K\right) ^{\vee
}\right) ^{\wedge }\left( 0_{\hat{G}}\right) \right\vert =\left\vert 
\underset{G}{\int }\left( T_{-u}K\right) ^{\vee }\left( y\right)
\left\langle 0_{\hat{G}},-y\right\rangle d\lambda \left( y\right) \right\vert
\end{equation*}%
\begin{equation}
=\left\vert \underset{G}{\int }\left( T_{-u}K\right) ^{\vee }\left( y\right)
d\lambda \left( y\right) \right\vert \leq C\left\vert A\right\vert ^{-\frac{1%
}{q}}\left\Vert K\right\Vert _{\left[ (p_{1}^{\prime };(p_{2}^{\prime
};(p_{3}^{\prime }\right] _{\theta }}  \tag{2.52}
\end{equation}%
for all $u\in \hat{G}$. Since $\frac{1}{p_{3}^{\prime }}<\frac{1}{%
p_{1}^{\prime }}+\frac{1}{p_{2}^{\prime }}$, then $-\frac{1}{q}<0$. Thus the
right side of (2.52) approaches to zero for $\left\vert A\right\vert
\rightarrow \infty $. This implies $K=0$. But this is a contradiction with
the assumption $K\neq 0$. \ Then the assumption $\frac{1}{p_{3}^{\prime }}<%
\frac{1}{p_{1}^{\prime }}+\frac{1}{p_{2}^{\prime }}$ is not true. Therefore
we conclude $\frac{1}{p_{3}^{\prime }}\geq \frac{1}{p_{1}^{\prime }}+\frac{1%
}{p_{2}^{\prime }}.$
\end{proof}

\begin{corollary}
Let $\frac{1}{p_{3}^{\prime }}<\frac{1}{p_{1}^{\prime }}+\frac{1}{%
p_{2}^{\prime }}.$ If $A$ is an automorphism of $G$ satisfying $\left\vert
A\right\vert >1.$ Then $\tilde{M}_{\theta }[(p_{1}^{\prime };(p_{2}^{\prime
};(p_{3}^{\prime }]$ $=\left\{ 0\right\} .$
\end{corollary}

\begin{proof}
Take any $K\in \tilde{M}_{\theta }[(p_{1}^{\prime };(p_{2}^{\prime
};(p_{3}^{\prime }]$. Let $\Psi \in L^{1}\left( G\right) $ such that $0\neq $
$\Psi \symbol{94}\in \ell ^{1}\left( \hat{G}\right) $. By Theorem 7, we have 
$\Psi \symbol{94}K\in $ $\tilde{M}_{\theta }[(p_{1}^{\prime };(p_{2}^{\prime
};(p_{3}^{\prime }]$. On the other hand, since $K\in \tilde{M}_{\theta
}[(p_{1}^{\prime };(p_{2}^{\prime };(p_{3}^{\prime }]$, $K$ is bounded
function. Then we have $\Psi \symbol{94}K\in \ell ^{1}\left( \hat{G}\right) $%
. Since $\Psi \symbol{94}K\in \ell ^{1}\left( \hat{G}\right) \cap $ $\tilde{M%
}_{\theta }[(p_{1}^{\prime };(p_{2}^{\prime };(p_{3}^{\prime }]$ and $\frac{1%
}{p_{3}^{\prime }}<\frac{1}{p_{1}^{\prime }}+\frac{1}{p_{2}^{\prime }}$, by
Proposition 8, we have $\Psi \symbol{94}K=0$. Also since $\Psi \symbol{94}$
is non-zero function, we obtain $K=\left\{ 0\right\} $. This completes the
proof.
\end{proof}

\end{document}